\documentclass[a4paper, 11pt]{article}
\usepackage[utf8]{inputenc}     
\usepackage[T1]{fontenc}
\usepackage{lmodern}
\usepackage{graphicx}
\usepackage[english]{babel}
\usepackage{subfigure}
\usepackage{amsmath, amsfonts, amssymb,amsthm}  
\usepackage{epstopdf}
\usepackage{textcomp}
\usepackage{mathrsfs}
\usepackage{multirow}
\usepackage{float}
\theoremstyle{plain}
\newtheorem{theorem}{Theorem}[section]
\newtheorem{proposition}[theorem]{Proposition}
\newtheorem{lemma}[theorem]{Lemma}
\newtheorem{corollary}[theorem]{Corollary}

\newtheorem{remark}[theorem]{Remark}

\theoremstyle{definition}
\newtheorem{definition}[]{Definition}
\newtheorem{example}[theorem]{Example}

\newcommand{\im}{\textbf{i}}
\newcommand{\bydef}{\,\stackrel{\mbox{\tiny\textnormal{\raisebox{0ex}[0ex][0ex]{def}}}}{=}\,}

\newcommand{\overa}{\overline{a}}

\newcommand{\ellnu}{\ell_{\nu}^1}

\begin{document}
\title{
Computer assisted proofs for transverse heteroclinics \\
by the parameterization method
}

\author{
J.D. Mireles James 
 \thanks{J.M.J partially supported by NSF 
grant DMS - 2307987.}
\footnote{Florida Atlantic University, 777 Glades Rd.,
Boca Raton, FL 33431, USA. jmirelesjames@fau.edu} \and
Maxime Murray  \footnote{Florida Atlantic University, 777 Glades Rd.,
Boca Raton, FL 33431, USA. mmurray2016@fau.edu} 
}

\maketitle

\begin{abstract}
This work develops a functional analytic framework for making
computer assisted arguments involving  
transverse heteroclinic connecting orbits between hyperbolic 
periodic solutions of ordinary differential equations. We exploit a Fourier-Taylor 
approximation of the local stable/unstable manifold of the periodic orbit, combined 
with a numerical method for solving two point boundary value problems via 
Chebyshev series approximations. The a-posteriori analysis developed provides 
mathematically rigorous bounds on all 
approximation errors, providing both abstract existence results and quantitative information 
about the true heteroclinic solution.  Example calculations are given for both the 
dissipative Lorenz system and the Hamiltonian Hill Restricted Four 
Body Problem.

\end{abstract}

\begin{center}
{\bf \small Key words.} 
{ \small parameterization method, heteroclinic connecting orbits, 
dissipative systems, Hamiltonian systems, 
computer assisted proof}
\end{center}

\section{Introduction} \label{sec:intro}
This paper describes constructive, computer assisted arguments
for proving theorems about cycle-to-cycle connecting orbits 
for ordinary differential equations. Our approach utilizes the parameterization 
method for stable/unstable manifolds attached to periodic 
orbits, as was developed in 
\cite{MR2551254,MR3118249,MR4130403,MR4030405,MR3304254,MR3871613,MR3754682}.
Validated numerical methods for 
Fourier, Taylor, and Chebyshev spectral approximations of these stable/unstable manifolds 
were developed (for example) in
\cite{MR0648529,MR0883539,MR2443030,MR2776917,HLM,MR3896998,MR3148084}.
(See also the references therein).

The ideas described here are part of an ongoing program
whose goal is to develop functional analytic tools for 
computer assisted proofs (CAPs) in nonlinear nonlinear analysis. 
The present work both builds on, and extends the
ideas of \cite{MR2821596,MR3353132,MR3207723,MR3906230,MR3022075,MR3281845}
on CAP for transverse heteroclinic and homoclinic connections between 
equilibrium solutions of vector fields, and the work of \cite{MR4658475} on 
functional analytic CAPs for transverse homoclinic connections for periodic orbits in 
Hamiltonian systems.  (Methods of CAP based on other methods
are discussed briefly in Remark \ref{rem:literature} below).

To facilitate the discussion, let $U \subset \mathbb{R}^d$ denote
a connected open set and $f \colon U \to \mathbb{R}^d$
be a real analytic vector field.  Suppose that $T > 0$, that 
$\lambda \in \mathbb{C}$ has nonzero real part, and that   
 $P \colon \mathbb{R} \times [-\tau,\tau] \to \mathbb{R}^d$ is a 
 smooth function, $T$-periodic in its first variable, taking values in 
 $U$. If $P$ solves the partial differential equation equation 
\begin{equation} \label{eq:parmMethod}
\frac{\partial}{\partial \theta} P(\theta, \sigma) 
+ \lambda \sigma \frac{\partial}{\partial \sigma} P(\theta, \sigma)
= f(P(\theta, \sigma)),
\end{equation}
then the image of $P$ is a local stable or unstable manifold 
for a periodic orbit of the vector field $f$.

To see this, note first 
that the image of $P$ is a cylinder, due to the 
periodicity of $P(\theta, \sigma)$ in its first variable.
Next observe that the equator of the cylinder is 
a periodic solution of the differential equation $x' = f(x)$.
To see this let $\gamma(\theta) = P(\theta, 0)$,
and evaluate Equation \eqref{eq:parmMethod} at $\sigma = 0$.
This gives that 
\[
\frac{d}{d \theta} \gamma(\theta) = f(\gamma(\theta)),
\]
as desired.

Assume now that $\mbox{real}(\lambda) < 0$.
(The unstable case of $\mbox{real}(\lambda) > 0$ is treated in a similar
fashion).  Define  
\[
x(t) = x(t, \theta, \sigma) = P(t + \theta, \sigma e^{\lambda t}), 
\]
with $t \geq 0$ and $\sigma \in [-\tau, \tau]$.
To see that $x(t)$ is on the stable manifold
of $\gamma$, one must show that  
$x(t)$ is an solution of the differential equation
and then consider the limit as $t \to \infty$.

Write $\bar \sigma = \sigma e^{\lambda t}$ and 
$\bar \theta = \theta + t$.  Then for $t > 0$ and
$\bar \sigma \in [-\tau, \tau]$, we have that 
$(\bar \theta, \bar \sigma)$ is in the domain of $P$,
and 
\begin{align*}
\frac{d}{dt} x(t) & = \partial_1 P(t + \theta, \sigma e^{\lambda t}) 
+ \lambda \sigma e^{\lambda t}  \partial_2 P(t + \theta, \sigma e^{\lambda t}) \\
&= \partial_1 P(\bar\theta, \bar \sigma) 
+ \lambda \bar \sigma  \partial_2 P(\bar, \bar \sigma) \\
& = f(P(\bar \theta, \bar \sigma)) \\
& = f(x(t)).
\end{align*}
Here we pass from the second to the third line by applying the 
invariance equation \eqref{eq:parmMethod}.

Now we note that 
\begin{align*}
\lim_{t \to \infty} \|x(t, \theta, \sigma) -\gamma(t)\| &= 
\lim_{t \to \infty} \| P(t + \theta, \sigma e^{\lambda t})  - P(t + \theta, 0) \| \\
& = 0,  
\end{align*} 
so that the image of $P$
is a local stable manifold as desired.
We refer to the curve $x(t) = x(t, \theta, \sigma)$ is a 
``whisker'' of the periodic orbit, and that the whisker has
asymptotic phase $\theta$.

Note also that 
\[
\xi(\theta) = \frac{\partial}{\partial \sigma} P(\theta, 0),
\]
is the normal vector bundle (eigenfunction) 
associated with the Floquet multiplier $\lambda$.
%
%
%
To see this, observe that 
\begin{align*}
\frac{\partial}{\partial \theta} P(\theta, \sigma) = f(P(\theta, \sigma)) 
- \lambda \sigma \frac{\partial}{\partial \sigma} P(\theta, \sigma), 
\end{align*}
so that by taking the partial with respect to $\sigma$, we have 
\begin{align*}
\frac{\partial}{\partial \sigma} \frac{ \partial}{\partial \theta}  P(\theta, \sigma)&=
\frac{\partial}{\partial \theta} \frac{\partial}{\partial \sigma} P(\theta, \sigma) \\ 
& = Df(P(\theta, \sigma)) \frac{\partial}{\partial \sigma} P(\theta, \sigma) 
- \lambda \frac{\partial}{\partial \sigma} P(\theta, \sigma) 
- \lambda \sigma \frac{\partial^2}{\partial \sigma^2} P(\theta, \sigma).
\end{align*}
Evaluating at $\sigma = 0$ and plugging in the definitions of $\xi(\theta)$
and $\gamma(\theta)$, we see that 
\[
\frac{d}{d \theta} \xi(\theta) + \lambda \xi(\theta) = Df(\gamma(\theta)) \xi(\theta), 
\]
which we recognize as the equation of first variation for the normal invariant 
vector bundle/eigenfunction associated with the Floquet multiplier $\lambda$.

Now suppose that 
$P, Q$ are parameterizations of the local unstable and stable manifolds
of a pair periodic orbits $\gamma_1$ and $\gamma_2$ and suppose 
that $y \colon [0, R] \to \Omega$ is an orbit segment which begins on the image
of $P$ and terminates on the image of $Q$.  Then the orbit of any point on $y$
accumulates to $\gamma_1$ in backward time, to $\gamma_2$
in forward time, and hence is a heteroclinic connection from 
$\gamma_1$ to $\gamma_2$.  
Such an orbit segment can be seen as a zero of the operator equation 
\[
F(y, \alpha, \beta, T) = \left(
\begin{array}{c}
y(t) - P(\alpha) - \int_0^t f(y(s)) \, ds   \\
Q(\beta) - \int_0^T f(y(s)) \, ds
\end{array}
\right),
\]
where $\alpha, \beta$ are coordinates on the stable/unstable manifolds.

Appropriate phase conditions, 
and transversality are discussed in Section \ref{Section:Orbit}, where
we project this functional equation, and the appropriate functional equations for
for the parameterized manifolds, into Banach spaces of rapidly decaying coefficient 
sequences.  The main point at the moment is that the parameterizations $P$ and 
$Q$ allow us to formulate two point boundary value problems (BVP) describing 
the connecting orbits, and that these BVPs can be studied using 
well established methods of computer assisted proof.  See for example 
\cite{MR1131109,MR1100582,MR1159927,MR1208187,MR2220064,
MR2174417,MR2173545,MR2821596,BDLM}.

In the remainder of the paper, we describe the method and its application 
in both dissipative and Hamiltonian systems (these require different phase
conditions).  We illustrate its utility for both polynomial 
and nonpolynomial systems.  
Our computational method and a-posteriori validation schemes
for the periodic orbits and their local stable/unstable manifolds
are discussed in Section \ref{Section:Manifold} with the 
validation of the connecting orbits discussed in Section 
\ref{Section:Orbit}.  Example applications are presented in 
Section \ref{sec:examples}.
A number of technical details are developed in the 
Appendices \ref{Radii}, 
\ref{sec:BanachSpaces}, and 
\ref{sec:CauchyBound}.
But before moving on to these developments,
 we first provide a few additional remarks and then 
introduce the main example applications studied in the 
present work.

\begin{remark}[Solving Equation \eqref{eq:parmMethod}] \label{rem:schemes}
{\em
There are at lest two common approaches to solving Equation \eqref{eq:parmMethod}.
The first is to treat it as an ``elliptic'' equation, 
and solve ``all at once'' via some iterative scheme. 
The second is to make a
power series ansatz
\[
P(\theta, \sigma) = \sum_{n=0}^\infty p_n(\theta) \sigma^n, 
\]
for the solution, and to derive the \textit{homological equations}
satisfied by the coefficient functions $p_n(\theta)$.  These homological 
equations turn out to be linear ODEs, and one can solve
them recursively.  When pursuing this later approach, the calculations
above determine the solution $P$ to first order.

It is the order-by-order approach which we adopt below, and we
note that it has been exploited in earlier 
computer assisted proofs for the parameterization method, in
a number of different contexts.  See for example
\cite{MR2821596,MR3207723,manifoldPaper1,MR4658475,MR3871613}
and also the lecture notes of \cite{MR3792792}.
Validated computer assisted error bounds for the
 ``elliptic'' or ``all-in-one-go'' approach of (i) are discussed 
 in detail in \cite{MR3906120,MR3735860}.  
}
\end{remark}

\begin{remark}[Literature on CAP for connecting orbits] \label{rem:literature}
{\em
To the best of our knowledge, the first computer assisted proof of chaotic
dynamics appeared in the early 1990's \cite{MR1236201}.  The 
paper considered the area preserving Taylor-Greene-Chirikov Standard map,
and the proof exploited the homoclinic tangle theorem of Smale \cite{MR0228014}.  
In this approach, the main steps in  
the argument involved establishing the existence of a 
transverse homoclinic intersection between 
the (one dimensional) stable and unstable manifolds of the 
map's hyperbolic fixed point.

By the late 1990's and early 2000's a number of authors had 
developed computer assisted arguments using topological tools  
-- like either Conley or Brower indices --
and applied them to prove theorems about chaotic dynamics in 
both discrete and continuous time dynamical systems.  For 
example the papers 
\cite{MR1394246,MR1485762,MR1429839} provided proofs of 
chaotic dynamics in the Henon map and in a Poincar\'{e} section of 
the R\"{o}ssler system.  Extensions to non-uniformly hyperbolic chaotic 
sets like in the Lorenz and Chua circuit are found in 
\cite{MR1453709,MR1626596,MR1661345,MR1276767,MR1459392,MR1808460}.
More recently, topological techniques for CAP have also been extended to 
infinite dimensional systems like population models with spatial 
dispersion \cite{MR2067140,MR3124898} and 
parabolic partial differential equations \cite{MR4113209}.

The methods just cited establish the existence of chaotic invariant 
horseshoes via direct topological covering arguments. 
The resulting sets may however not be attracting.  
Geometric/topological techniques for proving theorems
about chaotic attractors were first developed in 
\cite{MR1701385,MR1870856} to resolve Smale's 14th question.
Extensions of these methods are found in \cite{MR2736320,MR3906255}.

Finally we mention that computer assisted methods of proof for 
heteroclinic and homoclinic solutions of ODEs based on geometric/topological 
methods have a long history as well.  For example, the papers 
\cite{MR1870131,MR1862804} describe a geometric approach 
based on covering relations/cone conditions, and apply
this approach to the Henon systems. 
Extensions to some problems in Celestial mechanics
and to the Michelson system are given in 
\cite{MR1961956,MR2174417,MR2173545,MR2271217,MR2262261}.

A method which also leads to transversality 
is developed, and applied to problems in Celestial Mechanics
(similar to the problems considered here) in \cite{MR3032848}.
We also mention that topological/geometric methods for studying Arnold 
diffusion in Celestial Mechanics problems have been developed in 
\cite{MR4544807}.
A study of the connecting orbit structure of the Swift-Hohenberg 
PDE using Conley index/connection matrix techniques is 
in the work of \cite{MR2136516}.

Of course the preceding light review is far from comprehensive, and the interested reader
will find many additional references and a much more complete discussion by 
consulting the papers cited above.  For more information, we also refer
to the review articles of 
\cite{jpjbReview,MR3990999} and to the books of 
\cite{MR2807595,MR3822720,MR3971222}.
}
\end{remark}

\subsection{Two Example Problems}
The functional analytic setup for heteroclinic connections
in dissipative systems is a little different from the setup required in the 
Hamiltonian setting.  To illustrated each, 
we fix a particular dissipative and Hamiltonian example
problem.

\subsubsection{A Dissipative Example: The Lorenz System}
The Lorenz system is a three parameter family of quadratic vector fields
on $\mathbb{R}^3$ given by 
\begin{align} \label{eq:Lorenz}
x' &= \sigma(y - x) \\
y' & = x(\rho - z) - y \\
z' &= xy - \beta z
\end{align}
where the ``classic'' parameter values studied by 
Edward Lorenz in 1963 are $\rho = 28$, $\sigma = 10$, and $\beta = 8/3$ 
\cite{MR4021434}.
The system is an extremely simplified toy model of atmospheric convection and,
perhaps more importantly, 
has become an iconic example of a simple nonlinear system exhibiting 
rich dynamics.  

We note that the system is dissipative (contracts phase space volume in 
forward time) as can be seen by computing the determinant of the 
Jacobian matrix of the vector field.
So periodic solutions, when they exist, 
tend to be isolated.  More precisely, two periodic orbits may be very close together 
in the phase space, but in this case they will have very different periods.  
So, when computing periodic a solution the 
frequency or period of the orbit is one of the unknowns.

For many parameter values (including the classic parameters recalled above)
the dominant feature of the phase space of Equation \eqref{eq:Lorenz} is the so 
called \textit{Lorenz attractor}, illustrated
 in the top frame of Figure \ref{fig:LorenzAttractor}.  Qualitative properties of the dynamics on the 
 attractor are discussed in detail in 
 \cite{MR0461581,MR0556582,MR0682059,MR0718132}.
 A constructive proof of the existence of the Lorenz attractor 
 at the classic parameter values, along with direct verifications of many of 
 its properties was given by Warwick Tucker in 1999 
\cite{MR1701385,MR1870856}.

From the point of view of the present work, what is important is that hyperbolic
periodic orbits are dense in the attractor,  and that we should typically 
expect heteroclinic and homoclinic connections between them.  
The bottom frame of Figure \ref{fig:LorenzAttractor}
illustrates three of the shortest periodic orbits on/near the 
attractor, and highlights the fact that they do provide a skeleton of its shape.
These orbits are coded by $AB$, $AAB$, and $ABB$ where an $A$ or $B$
represents a wind around the left or right lobe of the attractor. 
As an illustration the utility of the techniques described in the present work, we 
prove the existence of these periodic orbits and establish the existence of
transverse heteroclinic connections between them.

\begin{figure}
\centering
\subfigure{\includegraphics[width=0.7\textwidth]{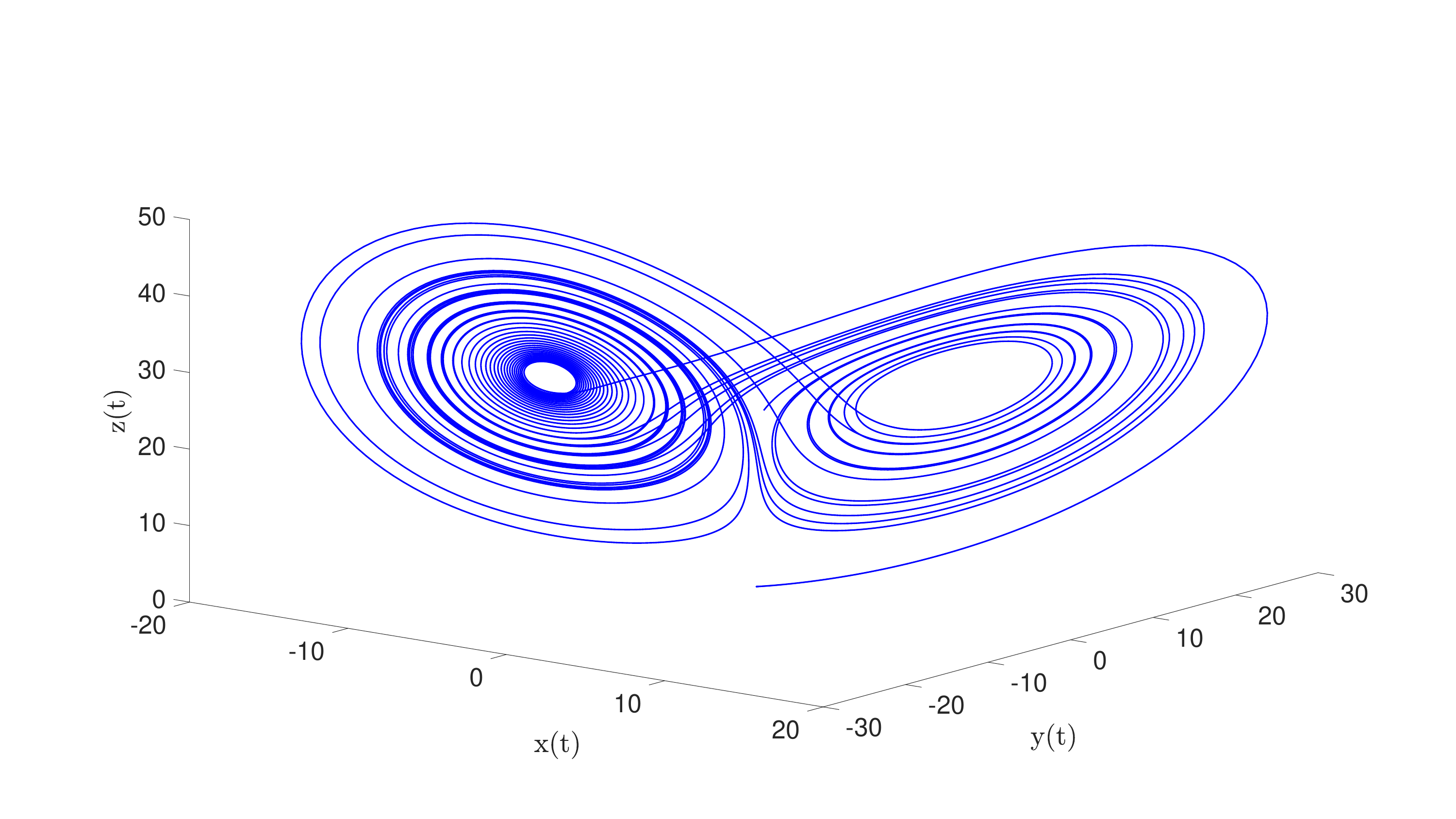}}
\subfigure{\includegraphics[width=0.7\textwidth]{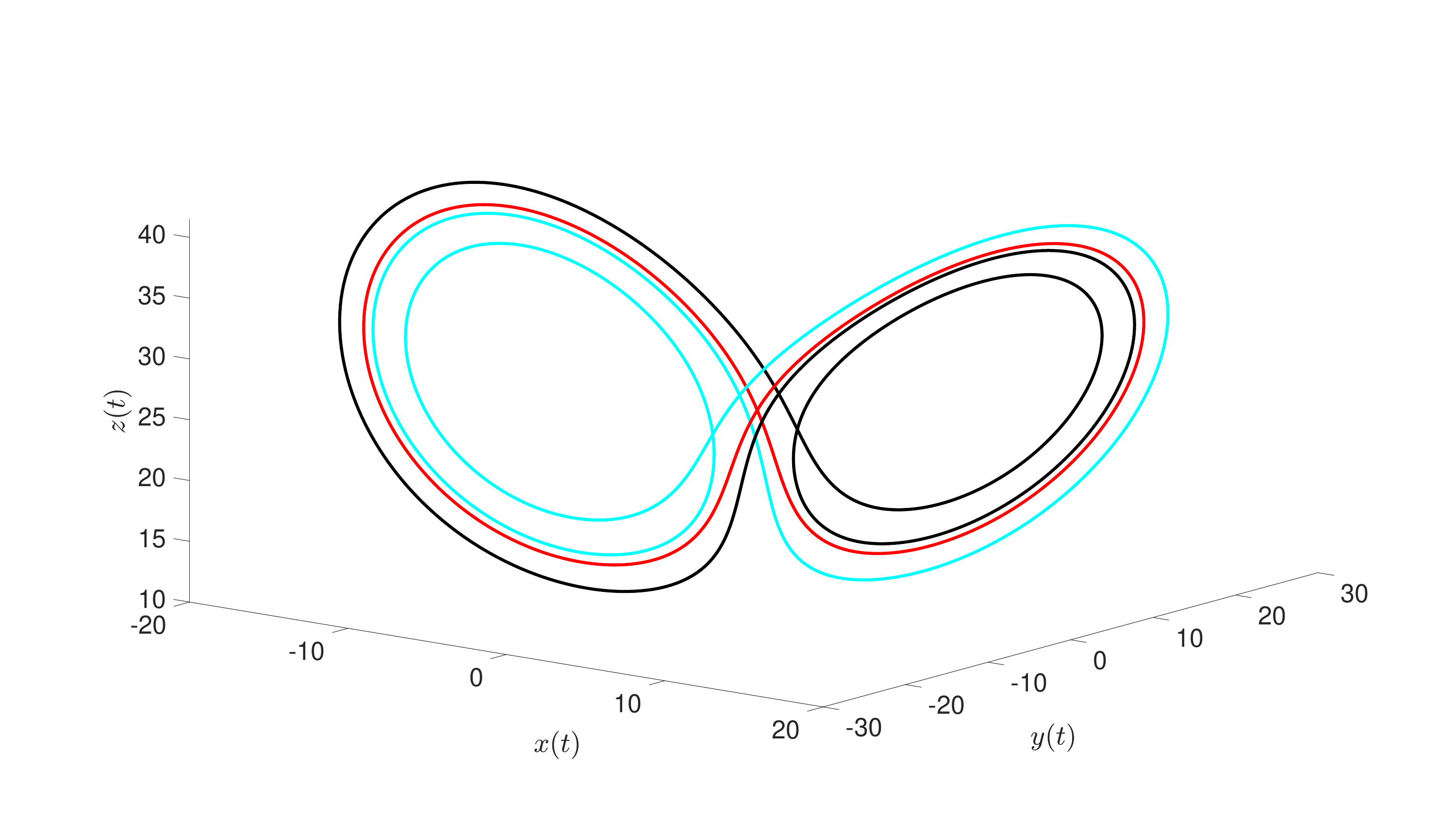}}
\caption{(Top) Illustration of the Lorenz attractor obtained by integrating a single 
initial condition.  Here we chose an initial condition near the origin, but this is 
incidental.  (Bottom) Three periodic orbits of the Lorenz system at the classic parameter values. 
Note that these three periodic orbits already give a strong impression of the shape of the 
Lorenz attractor.
}\label{fig:LorenzAttractor}
\end{figure}

\subsubsection{A Hamiltonian Example: A Lunar Hill Problem}
Many Hamiltonian systems are periodic orbit rich.  
In some cases this is because the system 
is a perturbation of an integrable system, and many periodic orbits 
will survive the perturbation.
In other cases, it may be that Poincare recurrence holds in some 
energy level set, so that periodic orbits are in fact dense.  
In any event, periodic orbits in systems preserving first integrals
generically occur in one parameter families or ``tubes'', and extra 
care is needed to isolate an orbit.  

We the system of differential equations given by 
\begin{align} \label{eq:Hill}
x' & = p \nonumber  \\
y' & = q \\
p' & =  \lambda_2 x + 2 q  - \frac{x}{\left(x^2+y^2 \right)^{3/2}} \nonumber  \\
q'  & =  \lambda_1 y - 2 p   - \frac{y}{\left(x^2+y^2 \right)^{3/2}}. \nonumber
\end{align}
Here
\[
\lambda_1 = \frac{3}{2}(1 - d), \quad \quad \quad \lambda_2 = \frac{3}{2} (1 + d), 
\]
and 
\[
d = \sqrt{1 - 3 \mu + 3 \mu^2}.
\]
We refer to Equation \eqref{eq:Hill} as the Hill equilateral restricted four 
body problem (HRFBP), and note that it reduces to the 
standard lunar Hill problem when $\mu = 0$.  
The HRFBP has the continuous symmetry 
\[
J(x,y,p,q) =  
\lambda_2 x^2 + \lambda_1 y^2  - p^2 - q^2
 + \frac{2}{\sqrt{x^2 + y^2}},
\]
which is referred to as the Jacobi integral.

Equation \eqref{eq:Hill} was derived by Burgos and Gidea in \cite{MR3346723},
and describes the dynamics of an infinitesimal particle (like a small rock or man-made satellite)
moving in the vicinity of a large astroid.  The astroid participates in a three body equilateral triangle
relative equilibrium (Lagrangian triangle), in the following sense.  
We assume that there are three massive gravitating bodies $m_1 \gg m_2 \gg m_3 > 0$, 
where $m_3$ is the mass of the astroid, and that these three bodies 
 orbit their common center of mass on Keplerian 
circles in such a way that - in an appropriate co-rotating frame - the three bodies form the vertices of an 
equilateral triangle.  The two larger massive bodies $m_1$ and $m_2$ are ``sent to infinity'' in such a way that 
their influence is still felt at $m_3$, which is located at the origin.  The parameter $\mu = m_1/(m_1+ m_2)$
describes the mass ratio of the two massive bodies at infinity.  Note that $\mu \in [0, 1/2]$ as
$\mu = 1$ when $m_2 = 0$ and $\mu = 1/2$ when $m_2 = m_1$.

This kind of configuration actually occurs
in our solar system, for example with $m_1$ the Sun, $m_2$ Jupiter, and $m_3$ a 
Trojan astroid like Hector,
so that the orbits of Equation \eqref{eq:Hill} would describe the dynamics of a satellite
maneuvering close enough to Hector that the effects of the Sun and Jupiter can be 
considered as small. 
Other configurations exist involving Jupiter and its moons.  

The system has four equilibrium solutions,
also known as \textit{libration points},  
two of which  are on the $x$-axis and are denoted by 
$L_1$ and $L_2$. 
These are located at 
\[
L_1 = \left(
\frac{1}{\lambda_2^{1/3}}
\right), \quad \quad \mbox{and} \quad \quad 
L_2 = \left(
\frac{-1}{\lambda_2^{1/3}}
\right).
\]
The other two equilibria located on the $y$-axis denoted by $L_3$ and $L_4$
and located at 
\[
L_3 = \left(
\frac{1}{\lambda_1^{1/3}}
\right), \quad \quad \mbox{and} \quad \quad 
L_4 = \left(
\frac{-1}{\lambda_1^{1/3}}
\right),
\]
play no role in the present work.

While the stability type of $L_3$ and $L_4$ depend on the mass ratio $\mu$, 
$L_{1,2}$ have saddle-center stability for all values of $\mu \in (0, 1/2]$.
The center eigenvalues at $L_1$ and $L_2$ give rise to a Lyapunov family of 
periodic orbits, locally parameterized by the Jacobi integral.  For energies near 
the energy of $L_{1,2}$, the Lyapunov orbits inherit one stable and one unstable 
Floquet multipliers from the stable/unstable directions of $L_{1,2}$, and hence
have attached stable/unstable manifolds.  The stable/unstable manifolds
of these periodic orbits could in turn intersect
in the energy level set, giving rise to chaotic dynamics. 

In Figure \ref{figure:LyapunovFamily}, we illustrate a number of 
periodic orbits in the Lyapunov families of $L_{1,2}$, for some selected values of the Jacobi integral. 
The small and large green orbits are at a the same energy levels, and 
establish the existence of heteroclinic connections
between some of these (and some others) below.

\begin{figure}
\centering
\includegraphics[width=0.9\textwidth]{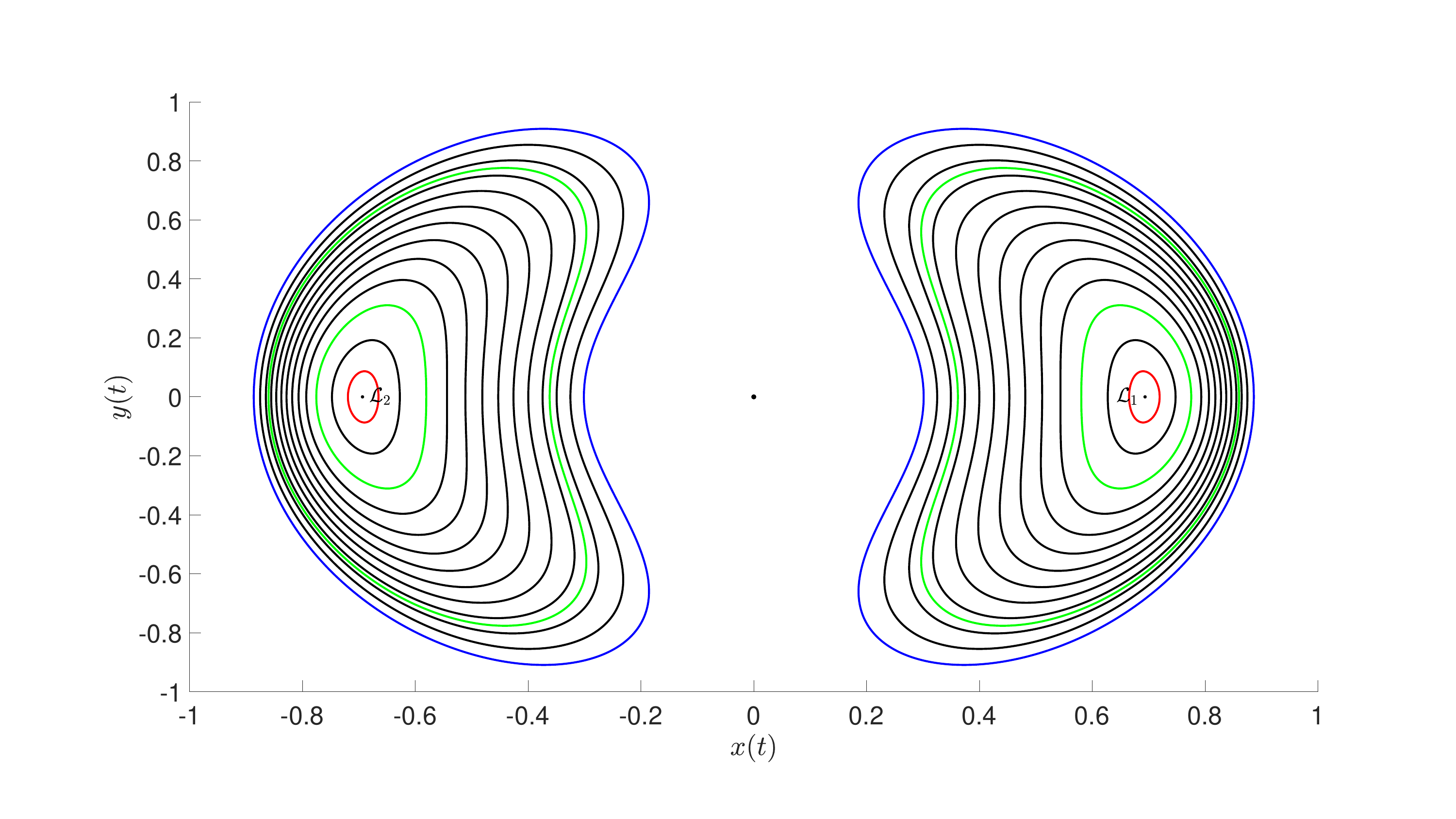}

\caption{Family of periodic orbits at $\mathcal{L}_1$ and $\mathcal{L}_2$ for $\mu=0.00095$. The orbits displayed have Jacobi integral between $C=2.00$ (in blue) and $C=4.30$ (in red). The periodic orbits at $C=4$ and $C=2.5$
(shown in green), are the basis of the computations discussed in Section \ref{sec:HillResults}.
That is, we will establish the existence of heteroclinic connections between these.} \label{figure:LyapunovFamily}
\end{figure}

\begin{remark}[Polynomial Embedding] \label{rem:poly}
{\em
The methods of the present work rely heavily on formal 
series manipulations which are much simpler when we 
work with polynomial nonlinearities.  For this reason, we will
derive a vector field which is somehow
equivalent to Equation \eqref{eq:Hill}, but which has 
only polynomial nonlinearities.  This can be done   
using an embedding technique discussed in detail in 
\cite{MR4208440,MR4292532,MR3906230}.

The idea is to introduce new variables for the non-polynomial 
terms, and append polynomial differential equations 
describing these nonlinearities.  Since the elementary 
functions of mathematical physics themselves solve 
simple (usually polynomial) differential equations, this idea applies
to many of the examples which arise in practice.  
The cost of this embedding is that we increase the dimension 
of the system by a number which depends on the 
complexity of the non-polynomial terms.

To see this idea at work in the HRFBP,
we begin by letting
\[
r = x^2 + y^2,
\]
so that
\[
z(r) = (x^2 + y^2)^{-1/2} = r^{-1/2}.
\]
Note that $z(r)$ generates the non-polynomial part of 
Equation \eqref{eq:Hill} in the following sense.
Differentiating obtains 
\begin{align*}
z' & = -\frac{1}{2} r^{-3/2} (2x' x+ 2y' y) \\  
& = -(r^{-1/2})^3 (x x' + y y') \\
&  = -z^3 (x x' + y y'). \\    
\end{align*}
Combining this with Equation \eqref{eq:Hill},
and making the appropriate substitutions, we have 
\begin{align} \label{eq:HillPoly}
x' & = p  \nonumber \\
y' & = q \nonumber \\
p' & =  \lambda_2 x + 2 q  - x z^3   \nonumber \\
q'  & =  \lambda_1 y - 2 p   - y z^3 \nonumber  \\
z' & = -z^3 (x p + y q).
\end{align}
Now, given initial conditions $(x_0, p_0, y_0, q_0) \in \mathbb{R}^4$,
we require that 
\[
z(0) = z_0 = \frac{1}{\left(x_0^2 + y_0^2 \right)^{1/2}}.
\]
Note that if $x$, $y$ are non-zero, then this constraint is 
equivalent to  
\[
\left(x_0^2 + y_0^2 \right)^{1/2} z_0 = 1,
\]
or, after squaring
\begin{equation}\label{eq:polyConstraint}
\left(x_0^2 + y_0^2 \right) z_0^2 = 1.  
\end{equation}

We can easily check that, for any orbit which does 
not pass through $x = y = 0$ (collision) 
the first four components of a
 solution of Equation \eqref{eq:HillPoly} provide a solution 
to Equation \eqref{eq:Hill}. 
Note that if $x = y = 0$ in Equation \eqref{eq:HillPoly}, then 
we cannot reverse the derivation of the constraint 
Equation \eqref{eq:polyConstraint}.  That is, the constraint
on the initial condition carries the singularity of the original 
problem (i.e. this polynomial embedding technique, despite first appearances,
is not a regularization technique).
}
\end{remark}

\section{Validated numerics for the parameterization method}\label{Section:Manifold}
The first step in our approach 
is to project the problems into appropriate spaces of 
Fourier, Fourier-Taylor, and Chebyshev series coefficients.  This reduces the ordinary 
and partial differential equations describing periodic orbits, normal vector 
bundles, parameterized invariant manifolds, and connecting orbits
to systems of infinitely many algebraic equations.
By truncating these algebraic equations and applying 
numerical Newton schemes, we obtain the approximate solutions
around which we construct our a-posteriori arguments.

The first step is to study the order zero and one terms in $P$, the solution of equation \eqref{eq:parmMethod} expressed using a power series in its second variable.  That is, 
the periodic orbit, its Floquet multiplier and associated normal bundle.
We begin by looking for $\gamma$ periodic having Fourier expansion
\[
\gamma(t) = \sum_{k \in \mathbb{Z}} a_{0,k} e^{i k \omega t},
\]
where $a_{0,k} \in \mathbb{C}^n$ and $\omega = 2 \pi/T$. 
Here, the zero in the subscript denotes the fact that the periodic orbit
is the zero-th order term in stable/unstable manifold expansion
described in Section \ref{sec:intro}.

In application problems we are interested in real valued  $\gamma$, 
so that $a_{0,k} = \mbox{conj}(a_{0,-k})$ for all $k \in \mathbb{Z}$. 
Moreover $\gamma$ is analytic (as $f$ is analytic in a neighborhood of $\Gamma = \{ \gamma(t) : t \in [0,T] \}$), 
and the Fourier coefficients decay exponentially fast, and
there is a $\nu > 1$ so that $\{a_k\}_{k \in \mathbb{Z}} \in (\ell_\nu^1)^n$.

Noting that 
\[
\frac{d}{dt} \gamma(t) = \sum_{k \in \mathbb{Z}} i \omega k a_{0,k} e^{i \omega k t},
\]
and letting
\[
f(\gamma(t)) = \sum_{k \in \mathbb{Z}} q_k e^{i \omega k t},
\]
where the $q_k$ depend on the $a_{0,k}$, we have that a periodic orbit can be 
viewed as a solution of the equation
\[
\sum_{k \in \mathbb{Z}} i \omega k a_{0,k} e^{i \omega k t} = 
\sum_{k \in \mathbb{Z}} q_k e^{i \omega k t}.
\]
Since two smooth periodic functions are equal if and only if the 
coefficients of their Fourier expansions are equals,
matching like-terms leads to the infinite system of algebraic equations
\[
 i \omega k a_{0,k} e^{i \omega k t} -  q_k(a_0) = 0,
 \quad \quad \quad \quad k \in \mathbb{Z},
\]
where $a_0 = \{a_{0,k} \}_{k \in \mathbb{Z}}$. 


The problem requires the addition of a phase condition to remove infinitesimal 
rotations and isolate a unique solution in the space of Fourier coefficients. 
In this work, we choose a relaxed Poincar\'{e} type condition by defining a
co-dimension one plane in $\mathbb{R}^n$ and requiring that $\gamma(0)$
is in this plane.  
More precisely, we choose $\bar \gamma_0 \in \mathbb{R}^3$ and 
$\bar \gamma_1= f(\bar\gamma_0)$, defining 
\begin{equation}\label{eq:PoincareOrbit}
F_0^0(a_0)= \bar\gamma_1 -\sum_{i=1}^n \left( \bar\gamma_0^i \cdot \sum_{|k|<K_0} a_{0,k}^i \right).
\end{equation}
Taking $K_0 \in \mathbb{N}$ yields a finite dimensional condition in Fourier space.

%
%
%
%
%
%

\begin{example}[Periodic solutions of the Lorenz system in Fourier space]\label{Ex:Orbit}
Finding a periodic orbit for the Lorenz system is equivalent finding a zero 
$(\omega, a_0^1, a_0^2, a_0^3) \in \mathbb{R} \times \ell_\nu^1 \times \ell_\nu^1 \times \ell_\nu^1$
 of the map $F_0 = (F_0^0, F_0^1, F_0^2, F_0^3)$ whose components are given by 
\begin{equation}\label{eq:OperatorOrbit}
F_0(\omega, a_0^1, a_0^2, a_0^3) =
\begin{pmatrix} 
F_0^0(a_0^1, a_0^2, a_0^3) \\
\left\{ - i \omega k a_{0,k}^1 +\sigma a_{0,k}^2 -\sigma a_{0,k}^1 : k \in \mathbb{Z} \right\} \\
\left\{ - i \omega k a_{0,k}^2 +\rho a_{0,k}^1 -a_{0,k}^2 - (a_0^1 \ast a_0^3)_k : k \in \mathbb{Z} \right\} \\
\left\{ - i \omega k a_{0,k}^3 -\beta a_{0,k}^3 +(a_0^1 \ast a_0^2)_k : k \in \mathbb{Z} \right\}
\end{pmatrix},
\end{equation}
where $F_0^0$ is as given in \eqref{eq:PoincareOrbit} with $n=3$. The equation $F = 0$ is studied in detail 
in \cite{HLM}, where a number of computer aided proofs for periodic orbits in the Lorenz system are 
implemented. In particular the truncated problem, the application of Theorem \ref{thm:contr}, 
and the operators $A$ and $A^{\dagger}$ introduced to compute the polynomial bound $Z(r)$ 
are discussed in detail. We employ these techniques in the present work as well.
\end{example}

\begin{example}[Periodic solutions of Hill's three-body problem in Fourier space]
In the case of Hill's three-body problem, we require two additional scalar phase conditions to isolate a unique solution. 
First, we impose the scalar constraint from Equation \eqref{eq:polyConstraint}, in order to guarantee
equivalence between the polynomial and non-polynomial problems.  Second, 
due to the fact that periodic solutions occur in one parameter families
parameterized by energy/frequency, we fix the value of the Jacobi integral to some 
$C \in \mathbb{R}$. Of course, to balance the resulting system, we must 
include two additional unfolding parameters into our zero finding problem.  
The unfolding parameters associated with the energy level and initial conditions
are discussed in detail in \cite{MR4208440,MR4576879}, and we use the 
same ideas below.

More precisely, set
\[
A_0^i \bydef \sum_{k \in\mathbb{Z}} a_{0,k}^i, \quad 1 \leq i \leq 5.
\]
Fixing $\bar{\gamma}_1, \bar{\gamma}_0^i \in \mathbb{R}^n$, for $i = 1, \ldots, 5$, 
and $C \in \mathbb{R}$ defining the desired energy level set,
the three phase conditions are encoded in the map $F_0^0$ defined by
\[
F_0^0(a_0^1,a_0^2,a_0^3,a_0^4,a_0^5) = \begin{pmatrix}
\bar\gamma_1 -\displaystyle \sum_{i=1}^5 \left( \bar\gamma_0^i \cdot \sum_{|k|<K_0} a_{0,k}^i \right) \\
(A_0^5)^2 \left[ \left(A_0^1 \right)^2 +\left( A_0^2 \right)^2 \right] - 1 \\
\lambda_2\cdot\left(A_0^1\right)^2 +\lambda_1\cdot\left(A_0^2\right)^2 -\left(A_0^3 \right)^2 -\left(A_0^4 \right)^2 +2\cdot A_0^5 -C
\end{pmatrix}.
\]
The additional scalar variables are included into the system through the map  
\[
G_0(\beta_1,\beta_2, a_0) = \begin{pmatrix}
0 \\
0 \\
0 \\
\left\{ \beta_1 a_{0,k}^1 : k \in \mathbb{Z} \right\} \\
\left\{ \beta_1 a_{0,k}^2 : k \in \mathbb{Z} \right\} \\
\left\{ \beta_1 a_{0,k}^3 : k \in \mathbb{Z} \right\} \\
\left\{ \beta_1 a_{0,k}^4 : k \in \mathbb{Z} \right\} \\
\left\{ \beta_1 a_{0,k}^5 +\beta_2(a_0^5 \ast a_0^5 \ast a_0^5)_k  : k \in \mathbb{Z} \right\}
\end{pmatrix}.
\]
Now our goal is to find $(\omega,\beta_1,\beta_2, a_0 ) \in \mathbb{R} \times \mathbb{R} \times \mathbb{R} \times (\ellnu)^5$ 
so that $F(\omega, \beta_1, \beta_2, a_0) = F_0(\omega, a_0) +G_0(\beta_1,\beta_2, a_0) = 0$, where $F_0$ is given by 
\begin{equation}\label{eq:OperatorOrbit}
F_0(\omega, a_0) =
\begin{pmatrix} 
F_0^0(a_0^1, a_0^2, a_0^3,a_0^4,a_0^5) \\
\left\{ - i \omega k a_{0,k}^1 +a_{0,k}^3 : k \in \mathbb{Z} \right\} \\
\left\{ - i \omega k a_{0,k}^2 +a_{0,k}^4 : k \in \mathbb{Z} \right\} \\
\left\{ - i \omega k a_{0,k}^3 +2a_{0,k}^4 +\lambda_2 a_{0,k}^1 -(a_0^1 \ast a_0^5 \ast a_0^5 \ast a_0^5)_k : k \in \mathbb{Z} \right\} \\
\left\{ - i \omega k a_{0,k}^4 -2a_{0,k}^3 +\lambda_1 a_{0,k}^2 -(a_0^2 \ast a_0^5 \ast a_0^5 \ast a_0^5)_k : k \in \mathbb{Z} \right\} \\
\left\{ - i \omega k a_{0,k}^5 -(a_0^1 \ast a_0^3 \ast a_0^5 \ast a_0^5 \ast a_0^5)_k -(a_0^2 \ast a_0^4 \ast a_0^5 \ast a_0^5 \ast a_0^5)_k : k \in \mathbb{Z} \right\}
\end{pmatrix}
\end{equation}
Moreover if $(\omega,\beta_1,\beta_2, a_0 )$ is a zero of $F$, then 
$\beta_1 = \beta_2 = 0$, so that we are actually solving Equation \eqref{eq:HillPoly}.
This can be shown precisely as in \cite{MR4208440,MR4576879}.
\end{example}

Next, we study the linear stability of the periodic orbit $\gamma(t)$ by finding an
invariant (stable/unstable) vector bundle $v(t)$ and associated Floquet exponent 
$\lambda \in \mathbb{C}$ for
$\gamma$.  The pair $(\lambda,v)$ solves the linear equation
\begin{equation}\label{eq:eigprob}
\dot{v}(t) + \lambda v(t) - A(t)v(t)=0,
\end{equation}
where $A(t) = Df(\gamma(t))$ is the differential of the vector field evaluated along the periodic orbit. 

In the present work we study periodic orbits with only one stable or unstable exponent.
Then the stable/unstable bundle provides a linear approximation of the local stable/unstable manifold 
attached to the periodic orbit $\gamma$. In the problems studied in this work, the stable/unstable
 bundles are periodic with period $T$. (This is the orientable case.  In the non-orientable case
 the bundle has period $2T$, a scenario not treated here but which would require only 
 a minor modification to the setting presented in the current section).

%
%
%
%
%

\begin{example}[Normal bundles in the example system]\label{Ex:Bundle}
In the case of the Lorenz system, we have that 
\begin{equation*}
A(t)=Df(\gamma(t)) =
\left( \begin{array}{c c c}
-\sigma & \sigma & 0 \\
\rho -\gamma^{(3)}(t) & -1 & - \gamma^{(1)}(t) \\
\gamma^{(2)}(t) & \gamma^{(1)}(t) & -\beta 
\end{array} \right).
\end{equation*}
We note that, after expanding each periodic function as a Fourier series, the coordinates with products can be written with the use of Cauchy-Convolution products, such as introduced in Definition \ref{def:CCproduct}. It follows that the Fourier coefficients $a_1= (a_1^1,a_1^2,a_1^3)$ of $v(t)$ satisfying \eqref{eq:eigprob} are a zero of
\[
F_1(\lambda,a_1)= \begin{pmatrix} 
F_1^0(a_1) \\
\left\{ - i \omega k a_{1,k}^1 +\sigma a_{1,k}^2 -\sigma a_{1,k}^1 : k \in \mathbb{Z} \right\} \\
\left\{ - i \omega k a_{1,k}^2 +\rho a_{1,k}^1 -a_{1,k}^2 - (a^1 \star a^3)_{1,k} : k \in \mathbb{Z} \right\} \\
\left\{ - i \omega k a_{1,k}^3 -\beta a_{1,k}^3 +(a^1 \star a^2)_{1,k} : k \in \mathbb{Z} \right\}
\end{pmatrix}.
\]
We note that the presentation use $a_0$, the Fourier coefficients of the periodic orbit, to present the operator with Cauchy-Convolution product. But the coefficients $a_0$ are fully determined by solving the previous operator introduced in Example \ref{Ex:Orbit}. The first coordinate, $F_1^0(a_1)$, is a scalar equation fixing the magnitude of the tangent bundle at $t=0$ to guarantee that the desired solution is well isolated. This scalar restriction also has the effect of balancing the unknown eigenvalue $\lambda$. This equation can be relaxed to a finite dimensional restriction. In this work, we take
\[
F_1^0(a_1) = \sum_{i=1}^3 \left( \sum_{|k|<k_0} a_{1,k}^i \right)^2 -L,
\]
where $L$ is the desired magnitude of the initial value of $v$. The case of Hill's three-body problem will be presented in the form $A(t)h(t)$ for some $h:[0,T] \to \mathbb{R}^5$ to simplify the presentation. The first four coordinates are 
\begin{equation*}
\bigg(A(t)h(t)\bigg)^i =
\begin{cases}
h^3(t), & \mbox{if}~ i=1, \\
h^4(t), & \mbox{if}~ i=2, \\
\lambda_2 h^1(t) +2h^4(t) -h^1(t)\gamma^5(t)\gamma^5(t)\gamma^5(t) -3\gamma^{1}(t)\gamma^5(t)\gamma^5(t)h^5(t),& \mbox{if}~ i=3, \\
\lambda_1 h^2(t) -2h^3(t) -h^2(t)\gamma^5(t)\gamma^5(t)\gamma^5(t) -3\gamma^{2}(t)\gamma^5(t)\gamma^5(t)h^5(t),& \mbox{if}~ i=4, \\
\end{cases}
\end{equation*}
and the last coordinate is given by
\begin{align*}
\bigg(A(t)h(t)\bigg)^5 = & -h^1(t)\gamma^{3}(t)\gamma^5(t)\gamma^5(t)\gamma^5(t) -h^2(t)\gamma^{4}(t)\gamma^5(t)\gamma^5(t)\gamma^5(t) \\
                         & -h^3(t)\gamma^{1}(t)\gamma^5(t)\gamma^5(t)\gamma^5(t) -h^4(t)\gamma^{2}(t)\gamma^5(t)\gamma^5(t)\gamma^5(t) \\
                         &-3\gamma^1(t)\gamma^{3}(t)\gamma^5(t)\gamma^5(t)h^5(t) -3\gamma^2(t)\gamma^4(t)\gamma^5(t)\gamma^5(t)h^5(t). \\
\end{align*}
So that the Fourier coefficients $a_1= (a_1^1,a_1^2,a_1^3,a_1^4,a_1^5)$ of the tangent bundle for a periodic solution of Hill's three-body problem with Fourier coefficients $a_0$, previously computed using Example \ref{Ex:Orbit}, will be a zero of
\begin{equation}
F_1(\lambda, a_1) =
\begin{pmatrix} 
F_1^0(a_1) \\
\left\{ - i \omega k a_{1,k}^1 +a_{1,k}^3 : k \in \mathbb{Z} \right\} \\
\left\{ - i \omega k a_{1,k}^2 +a_{1,k}^4 : k \in \mathbb{Z} \right\} \\
\left\{ - i \omega k a_{1,k}^3 +2a_{1,k}^4 +\lambda_2 a_{1,k}^1 -(a^1 \star a^5 \star a^5 \star a^5)_k : k \in \mathbb{Z} \right\} \\
\left\{ - i \omega k a_{1,k}^4 -2a_{1,k}^3 +\lambda_1 a_{1,k}^2 -(a^2 \star a^5 \star a^5 \star a^5)_k : k \in \mathbb{Z} \right\} \\
\left\{ - i \omega k a_{1,k}^5 -(a^1 \star a^3 \star a^5 \star a^5 \star a^5)_{1,k} -(a^2 \star a^4 \star a^5 \star a^5 \star a^5)_{1,k} : k \in \mathbb{Z} \right\}
\end{pmatrix}.
\end{equation}
The scalar equation $F_1^0(a_1)$ is similar to what was presented in the case of Lorenz and has the same effect on the problem.
\end{example}

The next step is to compute local stable/unstable manifold parameterizations via the Parameterization 
method of \cite{MR3118249,MR3467671,MR3304254}, as discussed in Section \ref{sec:intro}.
That is, we seek
$P: \mathbb{T}_T \times [-1,1]\to\mathbb{R}^n$, where $\mathbb{T}_T = \mathbb{R}\backslash [0,T]$ and $n=3$ or $n=5$ for the two examples of this work, 
a smooth function with 
\begin{equation*}
  P(\theta,0) = \gamma(\theta), \quad \theta \in \mathbb{T}_\tau,
\end{equation*}
with
\begin{equation}\label{eq:eigenvector}
 \frac{\partial}{\partial \sigma} P(\theta,0) = v(\theta), \quad \theta \in \mathbb{T}_\tau,
\end{equation}
solving the the invariance equation of Equation \eqref{eq:parmMethod}.

We develop the solution of \eqref{eq:parmMethod} as as Taylor expansion
\begin{equation}\label{eq:TaylorExpansion}
 P(\theta,\sigma) = \sum_{\alpha=0}^\infty A_\alpha(\theta)\sigma^\alpha,
\end{equation}
where each coefficient $A_\alpha$ is a periodic function with period $T$. 
Moreover, we take $P$ defined for $\sigma \in [-1,1]$.
This is done by choosing the value $L$, introduced in Example \ref{Ex:Bundle}, 
to be small or large enough, as this choice
directly effects both the size of the image of $[-1,1]$ 
and also the decay rate of the coefficients.  

Heuristically, 
choosing $L$ so that the coefficients decay in such a way that the 
last coefficient has norm on the order of machine precision guarantees
that  $[-1,1]$ will make a useful domain.  This is because, 
we typically find that the error of the parameterization on $[-1,1]$
is of the same order as the size of the last coefficient computed.
Again, these are only heuristics which must be validated 
via the a-posteriori analysis to follow.

Now, each Taylor coefficient is expanded as a Fourier series
\[
A_\alpha(\theta)= \sum_{k \in \mathbb{Z}} a_{\alpha,k}e^{i\omega k\theta},
\]
and, summarizing what has been done so far, we have that 
\begin{equation*}
 P(\theta,\sigma) = \sum_{\alpha=0}^\infty A_\alpha(\theta)\sigma^\alpha = \sum_{\alpha=0}^\infty\sum_{k\in \mathbb{Z}} a_{\alpha,k} e^{\frac{2\pi \im}{T}k\theta}\sigma^\alpha = \sum_{\alpha=0}^\infty\sum_{k\in \mathbb{Z}} a_{\alpha,k} e^{\im \omega k\theta}\sigma^\alpha,
\end{equation*}
where $a_{\alpha,k} \in \mathbb{C}^n$ for all $\alpha,k$, and 
\begin{align*}
\|P\|_\infty &= \sup_{(t,\sigma)\in \mathbb{T} \times [-1,1]} \left| \sum_{\alpha=0}^\infty \sum_{k\in \mathbb{Z}} a_{\alpha,k} e^{i\omega k\theta} \sigma^\alpha \right| \\
&\leq \sum_{\alpha=0}^\infty \|a_\alpha\|_{1,\nu} < \infty. 
\end{align*}
We also remark that the coefficients $a_{0,k}$ and $a_{1,k}$ are the Fourier coefficients of the periodic 
solution and the tangent bundle respectively. Both of which can be computed by following Examples 
\ref{Ex:Orbit} and \ref{Ex:Bundle} in the problems of interest. We introduce the following notation, which 
is used throughout the remaining presentation.

%
%
%
%
\begin{definition}
We set $a= (a^1,a^2,\hdots,a^n)$ to represent the Fourier-Taylor coefficients of a parameterized 
manifold. In the case treated in this work, the 
parameterized manifold will be analytic. The expression 
$a_\alpha$ will denote the $\alpha$-th order coefficient of $P$, so that $a_\alpha \in (\ellnu)^n$ 
for all $\alpha \geq 0$ and $a^i$ denotes one coordinate of the Taylor coefficients. We have that $a^i \in X_\nu$ for all $1\leq i \leq n$, with $X_\nu$ as presented in Definition \ref{def:FourierTaylorSpace}.
\end{definition}

Again, in the applications considered in the present work,
 the image of this parameterization is real and the coefficients must satisfy
\begin{equation}\label{eq:SymmetryManidols}
	a_{\alpha,-k} = \mbox{conj}\left(a_{\alpha,k}\right), \quad \mbox{for all} \; \alpha\geq 2 \; \mbox{and} \; k \geq 0.
\end{equation}
We remark that the case $k=0$ provides that $a_{\alpha,0} \in \mathbb{R}$ for all $\alpha$. 
The vector field of interest being polynomial, the power series for $P$ from Equation \eqref{eq:TaylorExpansion} 
can be substituted into the invariance equation \eqref{eq:parmMethod}.
It is then possible to rearrange the problem 
and solve order by order in the space of Taylor coefficients.  
This procedure leads to the \textit{homological equation} describing 
the $\alpha$-th Taylor coefficient
 for all $\alpha \geq 2$, $A_\alpha(\theta)$, and given by the 
 following ordinary differential equation with constant coefficients
%
%
%
%
\begin{align}\label{eq:coefficients}
 \frac{d A_\alpha (\theta)}{d \theta} + \alpha \lambda A_\alpha (\theta) & = (f \circ P)_\alpha (\theta) \\
 &= Df(a_0) A_\alpha + R(A)_\alpha,
\end{align}
which, since the nonlinear remainder $R(a)_\alpha$ will not depend on $A_\alpha$,
is a linear equation for $A_\alpha$.
It is instructive to see how this equation appears in examples.
%
%
%
%
\begin{example}[The Fourier operator for Taylor coefficients of higher order]\label{Ex:Higher}
In the case of the Lorenz system, we have that
\begin{align*}
(f \circ P)_\alpha (\theta) &= \begin{pmatrix} 
\sigma A_\alpha^2(\theta) - \sigma A_\alpha^1(\theta) \\
\rho A_\alpha^1(\theta) - A_\alpha^2(\theta) -\displaystyle \sum_{\beta=0}^\alpha A_\beta^1(\theta) A_{\alpha-\beta}^3(\theta)\\
- \beta A_\alpha^3(\theta) +\displaystyle \sum_{\beta=0}^\alpha A_\beta^1(\theta) A_{\alpha-\beta}^2(\theta) \\
\end{pmatrix} \\
&= DF(A_0(\theta))A_\alpha(\theta) +\begin{pmatrix} 
0 \\
-\displaystyle \sum_{\beta=1}^{\alpha-1} A_\beta^1(\theta) A_{\alpha-\beta}^3(\theta)\\
 \displaystyle \sum_{\beta=1}^{\alpha-1} A_\beta^1(\theta) A_{\alpha-\beta}^2(\theta) \\
\end{pmatrix}.
\end{align*}
This shows that, for any $\alpha$, it is only required to know the coefficients of lower order in order to solve the differential equation. Hence, we can solve the system exactly up to any given order $N$.

Again, we can write the problem as an infinite system of algebraic 
equations in Fourier space, whose zero is the unique sequence of Fourier-Taylor 
coefficients of $P$ up to a chosen order $N$. For each $\alpha\geq 2$, the Fourier coefficients $a_\alpha$ of $A_\alpha(\theta)$ are a zero of $F_\alpha : (\ellnu)^3 \to (\ell_{\nu'}^1)^3$ given by
\begin{equation}\label{eq:OperatorHigher}
F_\alpha(a_\alpha) =
\begin{pmatrix} 
\left\{ (-i \omega k -\alpha\lambda) a_{\alpha,k}^1 + \sigma a_{\alpha,k}^2 - \sigma a_{\alpha,k}^1 : k \in \mathbb{Z} \right\} \\
\left\{ (-i \omega k -\alpha\lambda) a_{\alpha,k}^2 + \rho a_{\alpha,k}^1 - a_{\alpha,k}^2 - (a^1 \star a^3)_{\alpha,k} : k \in \mathbb{Z} \right\} \\
\left\{ (-i \omega k -\alpha\lambda) a_{\alpha,k}^3 - \beta a_{\alpha,k}^3 + (a^1 \star a^2)_{\alpha,k} : k \in \mathbb{Z}\right\}
\end{pmatrix}.
\end{equation}
We note that this system is fully determined by the choice made in
 Examples \ref{Ex:Orbit} and \ref{Ex:Bundle}, so that no phase variable or 
 scalar equation is required. For the Hill's three-body problem, 
 $F_\alpha : (\ellnu)^5 \to (\ell_{\nu'}^1)^5$ is defined as
\begin{equation*}
F_\alpha(a_\alpha) =
\begin{pmatrix} 
\left\{ (-i \omega k -\alpha\lambda) a_{\alpha,k}^1 +a_{\alpha,k}^3 : k \in \mathbb{Z} \right\} \\
\left\{ (-i \omega k -\alpha\lambda) a_{\alpha,k}^2 +a_{\alpha,k}^4 : k \in \mathbb{Z} \right\} \\
\left\{ (-i \omega k -\alpha\lambda) a_{\alpha,k}^3 +2a_{\alpha,k}^4 +\lambda_2 a_{\alpha,k}^1 -(a^1 \star a^5 \star a^5 \star a^5)_{\alpha,k} : k \in \mathbb{Z} \right\} \\
\left\{ (-i \omega k -\alpha\lambda) a_{\alpha,k}^4 -2a_{\alpha,k}^3 +\lambda_1 a_{\alpha,k}^2 -(a^2 \star a^5 \star a^5 \star a^5)_{\alpha,k} : k \in \mathbb{Z} \right\} \\
\left\{ (-i \omega k -\alpha\lambda) a_{\alpha,k}^5 -(a^1 \star a^3 \star a^5 \star a^5 \star a^5)_{\alpha,k} -(a^2 \star a^4 \star a^5 \star a^5 \star a^5)_{\alpha,k} : k \in \mathbb{Z} \right\}
\end{pmatrix}.
\end{equation*}
\end{example}

After the series is computed to Taylor order $N$, 
the truncated approximation is equipped with validated error bounds
 using a fixed point argument, which will require some additional notation.
 \begin{definition}
Fix $N$ and define $X_\nu^N$ as
\[
X_\nu^N = \left\{ x \in X_\nu : \left\|x_\alpha\right\|_{1,\nu} = 0, \forall \alpha \geq N \right\}.
\]
We note that this defines a subspace of $X_\nu$.
\end{definition}
While the subspace $X_\nu^N$ just introduced is not closed under the 
Cauchy-Convolution product $\star$, we can truncate the products and
define polynomial operations from $X_\nu^N$ into itself. 
We set $F: \mathbb{R}^k \times (X_\nu^N)^n \to \mathbb{R}^k \times (X_{\nu'}^N)^n$ as
\[
F= \left( F_0, F_1, \hdots, F_{N-1},0,0,\hdots\right)
\]
using each coefficient $F_\alpha$ as defined in Examples \ref{Ex:Orbit}, \ref{Ex:Bundle}, 
and \ref{Ex:Higher} for $\alpha=0$, $\alpha=1$, and $2 \leq \alpha < N $ respectively. 
The number of scalar variables $k$ depends on the phase condition chosen to compute the 
periodic orbit and the tangent bundle. So that $k=2$ for Lorenz and $k=4$ for Hill's three-body 
problem.  We compute $\bar a^N$ a finite dimensional approximation of each of the 
$n\cdot N$ Fourier sequence and use the Radii polynomial approach, presented in Section \ref{Radii} ,
to obtain a constant $r^N >0$ such that 
\begin{align}\label{eq:BoundonCandX}
 \sum_{\alpha=0}^{N-1} \| a_\alpha - \bar a_\alpha \|_{1,\nu} \leq r^N.
\end{align}
Since the $\ellnu$ norm provides an upper bound on the $C^0$ norm, if follows that
\begin{align}
\|P - P^N\|_\infty &\leq \sum_{\alpha=0}^{N-1} \|A_\alpha - \bar A_\alpha\|_\infty + \left\| \sum_{\alpha=N}^\infty  A_\alpha \right\| \\
&\leq r^N + \left\| \sum_{\alpha=N}^\infty  A_\alpha \right\|. \label{eq:BoundTailUnknown}
\end{align}

For further detail regarding the application of the Radii polynomial approach, we refer the interested reader to \cite{MR3612178} and its application to similar examples in \cite{HLM,MR4658475,MR3545977}. Now, we aim to compute an upper bound for the remaining sum.

%
%
%
%
%
%
\subsection{Contraction Argument for the tail of the Fourier-Taylor expansion} \label{sec:radPolyParmMethod}

We first rewrite the invariance equation as a fixed-point problem
for $a_\alpha = F_\alpha(a)$ for all $\alpha \geq N$. We will see that, 
if $N \in \mathbb{N}$ is large enough, the resulting fixed point 
problem is a contraction near the zero tail.  It follows that there exists
a unique small tail so that the approximate solution plus this tail is the 
true invariant manifold parameterization.

Recall that the coefficients $a_\alpha$ 
satisfy the homological Equation \eqref{eq:coefficients}, which can be rewritten as
\begin{align}\label{eq:InvarianceFourier}
\mathcal{L}_\alpha(a_\alpha) +DF(a_0)a_\alpha = R_\alpha(a)
\end{align}
where $R_\alpha$ depends on the vector field of choice, 
and $\mathcal{L}_\alpha: (\ellnu)^n \to (\ell_{\nu'}^1)^n$ is given by
\[
\mathcal{L}_\alpha(a_\alpha)=  \bigg\{ (-i \omega k -\alpha\lambda)a_{\alpha,k} : k \in \mathbb{Z} \bigg\}, ~\forall \alpha.
\]
We introduce some definitions which simplify the presentation. The explicit 
definition of $R$ in the case of the Lorenz system and Hill's three body problem 
are given below.

\begin{definition}
We employ the decomposition 
$X_\nu= X_\nu^N + X_\nu^\infty$ where $X_\nu^N$ is as defined before, and 
\[
X_\nu^\infty = \left\{ a \in X_\nu : ~a_\alpha = 0, ~\forall \alpha < N \right\}.
\]
\end{definition}

\begin{definition}
For $x \in X_\nu^\infty$, define the norm
\[
\|x\|_{X_\nu^\infty} = \sum_{\alpha=N}^\infty \|x_\alpha\|_{1,\nu}.
\]
For $x \in (X_\nu^\infty)^n = (X_\nu^\infty,X_\nu^\infty,\hdots, X_\nu^\infty)$, we take
\[
\|x\|_{\left(X_\nu^\infty \right)^n} = \max_{1\leq i \leq n}  \|x^i\|_{X_\nu^\infty}.
\]
\end{definition}

%
%
%
%
\begin{example}
Recall the form of $f\circ P$ in Example \ref{Ex:Higher}, so that
$R_\alpha : X_\nu^3 \to (\ellnu)^3$, the right-hand side of 
\eqref{eq:InvarianceFourier} in the case of the Lorenz system, has
\begin{align}
R_\alpha(a) = \begin{pmatrix}
0 \\
\displaystyle\sum_{\substack{\alpha_1 + \alpha_2 = \alpha\\ \alpha_1,\alpha_2 > 0}} \sum_{\substack{k_1+k_2 = k\\ k_1,k_2 \in \mathbb{Z}}} a_{\alpha_1,k_1}^{1}a_{\alpha_2,k_2}^{3} \\
-\displaystyle \sum_{\substack{\alpha_1 + \alpha_2 = \alpha\\ \alpha_1,\alpha_2 > 0}} \sum_{\substack{k_1+k_2 = k\\ k_1,k_2 \in \mathbb{Z}}} a_{\alpha_1,k_1}^{1}a_{\alpha_2,k_2}^{2}
\end{pmatrix} \bydef \begin{pmatrix}
0 \\
0 \\
(a^1 \hat \star a^3)_\alpha \\
-(a^1 \hat \star a^2)_\alpha
\end{pmatrix}.
\end{align}
Note that each summation is the Cauchy-Convolution product from which the zeroth order terms have been 
removed and will be denoted $\hat \star$ to follow Definition \eqref{def:CCproduct}. In the case of Hill's three body problem, we have 
$R_\alpha : X_\nu^5 \to (\ellnu)^5$ and
\begin{align}
R_\alpha(a) =  -\begin{pmatrix}
0 \\
0 \\
(a^1 \hat \star a^5 \hat \star a^5 \hat \star a^5)_\alpha \\
(a^2 \hat \star a^5 \hat \star a^5 \hat \star a^5)_\alpha \\
(a^1 \hat \star a^3 \hat \star a^5 \hat \star a^5 \hat \star a^5)_\alpha 
+ (a^2 \hat \star a^4 \hat \star a^5 \hat \star a^5 \hat \star a^5)_\alpha
\end{pmatrix}.
\end{align}
\end{example}

Now, we note that $\mathcal{L}_\alpha$ is invertible for all $\alpha$, and that the 
formal inverse is 
\[
\mathcal{L}_\alpha^{-1}(a_\alpha)= 
\left\{ \frac{a_{\alpha,k}}{(-i \omega k -\alpha\lambda)} : k \in \mathbb{Z} \right\}.
\]
Let $I_{(\ellnu)^n}$ denote the identity operator on $(\ellnu)^n$. Applying 
$\mathcal{L}_\alpha^{-1}$ to both sides of \eqref{eq:InvarianceFourier} yields
%
%
%
%
\begin{align}\label{eq:InvarianceSmoothed}
\left[ I_{(\ellnu)^n} + \mathcal{L}_\alpha^{-1} \circ DF(a_0) \right] a_\alpha = \mathcal{L}_\alpha^{-1} \circ R_\alpha(a).
\end{align}
This suggests that the inverse of the left-hand side can be expressed as a 
Neumann series. Indeed, this works for large values of $\alpha$, and the finitely 
many cases where it fails are treated by further splitting the problem and employing
a numerical approximate inverse. This argument is completed in the remainder of this section
and the results are summarized in the following theorem.
%
%
%
%
\begin{theorem}\label{Thm:InverseG}
For $\alpha \geq N$, consider
\[
\left[ I_{(\ellnu)^n} +\mathcal{L}_\alpha^{-1} D\mathcal{F}_0(a_0)\right] a_\alpha = \mathcal{L}_\alpha^{-1}R_\alpha(a) .
\]
Then, for all $\alpha \geq N$, there exists a invertible and linear operator $\mathcal{G}_{\alpha}: (\ellnu)^n \to (\ellnu)^n$ such that  $a_\alpha$ satisfies
\[
a_\alpha  = \mathcal{G}_\alpha^{-1} R_\alpha(a),~\mbox{for all}~ \alpha \geq N,
\]
and the operator $\mathcal{G}_\alpha^{-1} R_\alpha :  (\ellnu)^n \to (\ellnu)^n $ is Fr\'{e}chet differentiable for each $\alpha$.
\end{theorem}
%
%
%
%
%
%
The differentiability of the composition follows directly from the 
linearity of $\mathcal{G}_\alpha^{-1}$, and the definition of $R_\alpha$. 
The desired result follows directly whenever 
\[
\left\|  \mathcal{L}_\alpha^{-1} DF(a_0) \right\|_{B(\left(\ellnu \right)^n)}<1,
\] 
in which case $\mathcal{G}_\alpha = I_{(\ellnu)^n} +\mathcal{L}_\alpha^{-1} D\mathcal{F}_0(a_0)$,
 and we invert using a Neumann series argument. Both operators $\mathcal{L}_\alpha^{-1}$ 
 and $DF(a_0)$ can be bound explicitly, and the condition is verified with computer
 assistance. After a bound is obtained for both operators, we can locate
 the first value $\alpha$ for which the Neumann condition is satisfied. 
Note that
\[
\left\|  \mathcal{L}_\alpha^{-1} \right\|_{B(\left(\ellnu \right)^n)}
\leq \frac{1}{\left|\alpha\mbox{Re}(\lambda)\right|},
\]
and there exist $K>0$ so that $\left\| DF(a_0) \right\|_{B(\left(\ellnu \right)^n)}  < K$.
Then for all $\alpha > \frac{K}{\left|\mbox{Re}(\lambda)\right|}$, 
the left hand side of equation \eqref{eq:InvarianceSmoothed} is inverted using a Neumann series
as desired. In this case, it follows that
\[
\left\| \mathcal{G}_\alpha^{-1} \right\|_{B(\left(\ellnu \right)^3)} < \frac{1}{1-\frac{K}{\left|\alpha \mbox{Re}( \lambda)\right|}} < \frac{1}{1-\frac{K}{M\left|\mbox{Re}( \lambda)\right|}},
\]
where $M$ denotes the ceiling of $\frac{K}{\left| \mbox{Re}(\lambda) \right|}$, 
the first value for which the condition is met. 

While this argument yields the desired conclusion 
for infinitely many values of $\alpha$, it might not work for all the values of interest. Indeed, 
it is possible that $M>N$. Let the set 
\[
\mathcal{A} = \left\{ \alpha \in \mathbb{Z} ~\vert~  N \leq \alpha < M  \right\},
\] 
collect the multi-indices of each Taylor coefficients in the tail so that the previous argument fails. 
If this set is empty, we simply rewrite the invariance equation into the form 
$a_\alpha= F_\alpha(a)$ for all $\alpha \geq N$ and no further work is required to 
prove Theorem \ref{Thm:InverseG}. 

If the set is non-empty, an alternative argument is introduced to justify that the invariance
 equation can always take an equivalent form $a_\alpha= \mathcal{G}_\alpha^{-1}R_\alpha(a)$.  
To this end, for each order in $\mathcal{A}$, we use the numerically computed
approximate inverse to obtain a bound on the norm of the inverse in a case-by-case manor. 
First, we use the numerical approximation to write $DF(a_0)$ as a sum. 
We set $a_0^\infty = a_0 - \bar a_0$, so that
\[
\| a_0^\infty \|_{1,\nu} \leq r^N.
\]
Note that $\bar a_0$ contains non-zero terms only up to some finite order, 
say $m$. We take advantage of this, and the fact that $F$ is polynomial to 
construct $D_\alpha^m,D_\alpha^\infty :(\ellnu)^n \to (\ellnu)^n$ an eventually zero 
operator that approximates $\mathcal{L}_\alpha^{-1}DF(\bar a_0)$ 
up to order $m$ and $D_\alpha^\infty $ its complement. Then, 
we choose an $\epsilon_{\alpha,1}$ having
\[
\left\| D_\alpha^\infty \right\|_{B((\ellnu)^n)} \leq  \epsilon_{\alpha,1}. 
\]
The following example illustrates the idea in the case of the Lorenz system. 
%
%
%
%
%
%
%
\begin{example}
Since we seek
\[
 \mathcal{L}_\alpha^{-1} \circ DF(a_0)h = D_\alpha^mh +D_\alpha^\infty h, ~\forall h \in \left( \ellnu\right)^3,
\]
let us write $h=(h^1,h^2,h^3)$, so that
\[
\left(\mathcal{L}_\alpha^{-1} \circ DF( a_0)h \right)_k = \begin{pmatrix}
\frac{\sigma}{(-i \omega k -\alpha\lambda)} (h_k^2-h_k^1) \\
\frac{1}{(-i \omega k -\alpha\lambda)}\left[ \rho h_k^1 -h_k^2 -(h^1 * a_0^3 )_k +(h^3 * a_0^1)_k \right] \\
\frac{1}{(-i \omega k -\alpha\lambda)}\left[ -\beta h_k^3 +(h^1 * a_0^2)_k +(h^2 * a_0^1)_k \right]
\end{pmatrix}.
\] 
Let $\bar h$  represent the truncation of $h$ to the same order as $\bar a_0$, 
and $\tilde h = h - \bar h$. Define
\begin{align*}
\left( D_\alpha^m h \right)_k = \begin{pmatrix}
\frac{\sigma}{(-i \omega k -\alpha\lambda)} (\bar h_k^2- \bar h_k^1) \\
\frac{1}{(-i \omega k -\alpha\lambda)}\left[ \rho \bar h_k^1 - \bar h_k^2 -( \bar h^1 * \bar a_0^3 )_k +( \bar h^3 * \bar a_0^1)_k \right] \\
\frac{1}{(-i \omega k -\alpha\lambda)}\left[ -\beta \bar h_k^3 +( \bar h^1 * \bar a_0^2)_k +(\bar h^2 *\bar a_0^1)_k \right]
\end{pmatrix}.
\end{align*}
We note that, by construction, $\left( D_\alpha^m h \right)_k = 0$ if $|k|>2m$. Moreover
\begin{align*}
\left( D_\alpha^\infty h \right)_k = &\begin{pmatrix}
\frac{\sigma}{(-i \omega k -\alpha\lambda)} ( \tilde h_k^2- \tilde h_k^1) \\
\frac{1}{(-i \omega k -\alpha\lambda)}\left[ \rho \tilde h_k^1 
- \tilde h_k^2 -( \tilde h^1 * \bar a_0^3 )_k +( \tilde h^3 * \bar a_0^1)_k \right] \\
\frac{1}{(-i \omega k -\alpha\lambda)}\left[ -\beta \tilde h_k^3 
+( \tilde h^1 * \bar a_0^2)_k +(\tilde h^2 *\bar a_0^1)_k \right]
\end{pmatrix} \\
&+ \begin{pmatrix}
0 \\
\frac{1}{(-i \omega k -\alpha\lambda)}\left[ -( \tilde h^1 * (a_0^\infty)^3 )_k
 +( \tilde h^3 * (a_0^\infty)^1)_k \right] \\
\frac{1}{(-i \omega k -\alpha\lambda)}\left[ ( \tilde h^1 * (a_0^\infty)^2)_k 
+( \tilde h^2 * (a_0^\infty)^1)_k \right]
\end{pmatrix}.
\end{align*}
The first term needs careful consideration for terms of lower order, 
while the term on the second line is bounded using Banach algebra estimates.

So, let 
\[
d_\alpha^\infty = \sup_{|k| \geq m} \left| \frac{1}{(-i \omega k -\alpha\lambda)}  \right|, ~\mbox{and}~ d_\alpha = \sup_{k \in \mathbb{Z}} \left| \frac{1}{(-i \omega k -\alpha\lambda)}  \right|.
\]
We use the dual estimates from Corollary 3 of Section 4 of \cite{HLM}
to compute $v_k$, bounding the $k-$th entry of each convolution product of order $|k|<m$. 
That is, it satisfies
\begin{align*}
\left| ( \tilde h^1 * \bar a_0^3 )_k \right| + \left|( \tilde h^3 * \bar a_0^1)_k\right| &\leq v_k,~\mbox{and} \\
\left| ( \tilde h^1 * \bar a_0^2)_k\right| + \left|(\tilde h^2 *\bar a_0^1)_k\right| &\leq v_k.
\end{align*}
This allows to bound the norm of the first term up to finite order $m$.
The tails of the summation are then bounded using
\[
\bar K \bydef \max \left\{ 2\sigma,\rho +1 
+\|\bar a_0^3\|_{1,\nu} +\|\bar a_0^1\|_{1,\nu}, \beta 
+\|\bar a_0^2\|_{1,\nu} +\|\bar a_0^1\|_{1,\nu} \right\}. 
\]
Finally, gathering the estimates together, we have that
\begin{align*}
\left\| D_\alpha^\infty \right\|_{B((\ellnu)^3} &\leq  d_\alpha^\infty\cdot \bar K
 +d_\alpha \cdot 2r^N + \sum_{|k|<m} \frac{v_k}{\left| -i \omega k -\alpha\lambda \right|} \nu^{|k|}  \\
&= \epsilon_{\alpha,1}.
\end{align*}

We remark that the bound on the first term is a bound on $DF(\bar a_0)$, 
and that this generates a bound similar to the
$\mathcal{Z}_1$ for the computer-assisted validation of the periodic orbit. 
This approach applies to the Hill three body problem, 
however the higher degree of the nonlinearities leads to even lengthier estimates,
which we suppress. 
\end{example}

The term $D_\alpha^mh$ is finite dimensional, so that $I_{(\ellnu)^3} + D_\alpha^m$ 
is eventually diagonal and its invertibility depends only on the invertibility of the finite part, 
which can be verified with computer assistance. 
Let $D_\alpha^\dagger$ denote the numerical approximation of the desired inverse, 
so that it is possible to use interval arithmetic to find a positive 
constant $\epsilon_{\alpha,2}$ such that
\[
\left\| \mathcal{E}_\alpha \right\|_{B((\ellnu)^n)} < \epsilon_{\alpha,2},
\]
where $\mathcal{E}_\alpha : (\ellnu)^n \to (\ellnu)^n$ is the error arising from
 the numerical inverse, defined by
\[
\mathcal{E}_\alpha= I_{(\ellnu)^n} - D_\alpha^\dagger \left( I_{(\ellnu)^n} +D_\alpha^m \right).
\]

Rewriting equation \eqref{eq:InvarianceSmoothed}, we have that
\begin{align*}
D_\alpha^\dagger \circ \mathcal{L}_\alpha^{-1} \circ R_\alpha(a) &= D_\alpha^\dagger\left[ I_{(\ellnu)^3} +\mathcal{L}_\alpha^{-1} DF(a_0) \right] a_\alpha \\
&= D_\alpha^\dagger\left[ I_{(\ellnu)^3} + D_\alpha^m +D_\alpha^\infty \right] a_\alpha \\
&= \left[ D_\alpha^\dagger\left( I_{(\ellnu)^3} + D_\alpha^m \right) +D_\alpha^\dagger D_\alpha^\infty \right] a_\alpha \\
&= \left[ I_{(\ellnu)^3} -\mathcal{E}_\alpha  +D_\alpha^\dagger D_\alpha^\infty \right] a_\alpha \\
&= \left[ I_{(\ellnu)^3} - \left( \mathcal{E}_\alpha  - D_\alpha^\dagger D_\alpha^\infty \right) \right] a_\alpha \\
\end{align*}
This formulation provides a condition which is verified using the estimates above, and 
completing the proof of Theorem \ref{Thm:InverseG}. We note that
\[
\left\| \mathcal{E}_\alpha  - D_\alpha^\dagger D_\alpha^\infty \right\|_{B((\ellnu)^n)} 
\leq \epsilon_{\alpha,2} + \epsilon_{\alpha,1} \cdot \left\| D_\alpha^\dagger \right\|_{B((\ellnu)^n)},
\]
but $D_\alpha^\dagger $ is eventually diagonal, so that its norm is bound via Corollary \ref{corollary:BoundEventuallyDiagonal}, and the condition
\begin{align}\label{eq:CriteriaFinite}
\epsilon_{\alpha,2} + \left\| D_\alpha^\dagger \right\|_{B((\ellnu)^3)}
\cdot \epsilon_{\alpha,1} < 1,
\end{align}
is verified with computer assistance. It follows that 
\[
\left[ I_{(\ellnu)^3} - \left( \mathcal{E}_\alpha  - 
D_\alpha^\dagger D_\alpha^\infty \right) \right] a_\alpha = 
D_\alpha^\dagger \circ \mathcal{L}_\alpha^{-1} \circ R_\alpha(a),
\]
can be rewritten in the form $a_\alpha = F_\alpha(a)$ for all $\alpha \in \mathcal{A}$ if \eqref{eq:CriteriaFinite} is satisfied. Hence, for all $\alpha \in \mathcal{A}$,
 if \eqref{eq:CriteriaFinite} holds then
\[
I_{(\ellnu)^3} - \left( \mathcal{E}_\alpha  - D_\alpha^\dagger D_\alpha^\infty \right)
\]
is invertible and
\[
\left\| \left[ I_{(\ellnu)^3} - \left( \mathcal{E}_\alpha  - 
D_\alpha^\dagger D_\alpha^\infty \right)\right]^{-1} \right\|_{B((\ellnu)^3)}
 < \frac{1}{1-\epsilon_{\alpha,2} - \left\| D_\alpha^\dagger 
 \right\|_{B((\ellnu)^3)} \cdot \epsilon_{\alpha,1}}.
\]

If this argument succeeds for all values $\alpha \in \mathcal{A}$, then the proof of Theorem \ref{Thm:InverseG} is complete with 
\[
\mathcal{G}_\alpha^{-1} = \left[ I_{(\ellnu)^3} - \left( \mathcal{E}_\alpha  
- D_\alpha^\dagger D_\alpha^\infty \right)\right]^{-1} 
D_\alpha^\dagger \mathcal{L}_\alpha^{-1},
\]
for $\alpha \in \mathcal{A}$. We note that the proof provides
\begin{align}\label{eq:BoundInverse}
\left\| \mathcal{G}_\alpha^{-1} \right\|_{B((\ellnu)^n)} \leq 
\max \left( \frac{1}{M|\lambda|-K} , \max_{\alpha \in \mathcal{A}}
 \left[ \frac{\left\| D_\alpha^\dagger \right\|_{B((\ellnu)^3)}}{|\alpha\lambda| 
 \left( 1-\epsilon_{\alpha,2} - \left\| D_\alpha^\dagger \right\|_{B((\ellnu)^3)} 
 \cdot \epsilon_{\alpha,1} \right)}  \right]\right) \bydef B_g,  
\end{align}
for all $|\alpha| > N$. 

We now use Theorem \ref{Thm:InverseG} to formulate a fixed 
point argument and complete the computation of the truncation error associated to 
a finite dimensional truncation of the parameterization $P$. 
The centerpiece of the argument is to use Theorem \ref{Radii} in the 
case $\bar a = 0$.  This application combines Theorems \ref{thm:contr}  and \ref{Thm:InverseG}.

\begin{corollary}\label{cor:ResultTail}
Let $a^N \in X^N$ denote the Fourier-Taylor coefficients of the solution of equation \eqref{eq:InvarianceFourier} for all $\alpha < N$, and 
define $T: X_\nu^\infty \to X_\nu^\infty$ by
\[
T(x)= \left\{ \mathcal{G}_\alpha R_\alpha (a^N +x) : |\alpha| \geq N \right\}, \forall x \in X_\nu^\infty.
\] 
Assume that $Y$ is a positive constant and $Z:(0,r_*) \to [0,\infty)$ is a non-negative function satisfying
\begin{align*}
\sum_{\alpha=N}^\infty \left\|R_\alpha(0) \right\|_{(\ellnu)^n}&\leq Y \\
\sup_{x \in \overline{B_r(0)}} \sum_{\alpha=N}^\infty \left\|DR_\alpha(x) 
\right\|_{\mathcal{B}((\ellnu)^n)} &\leq Z(r), \quad \mbox{for all} ~r \in (0,r_*) \\
\end{align*}
If there exists $r^\infty \in (0,r_*)$ such that $ B_G\left( Y +Z(r^\infty)r^\infty \right) <r^\infty$, 
then there exists $a^\infty \in B_{r^\infty}(0) $ such that $T(a^\infty) = a^\infty$.
\end{corollary}

Note that the estimates in Corollary \ref{cor:ResultTail} follows directly 
from the fact (established above) that for all $b,c \in X^\infty$
\begin{align*}
\|T(b)\|_{X_\nu^\infty} = \sum_{\alpha=N}^\infty \left\| \mathcal{G}_\alpha^{-1} 
R_\alpha(b) \right\|_{(\ellnu)^n} \leq B_g \cdot 
\sum_{\alpha=N}^\infty \left\|R_\alpha(b) \right\|_{(\ellnu)^n},
\end{align*}
and
\begin{align*}
\|DT(b)c\|_{X_\nu^\infty} = \sum_{\alpha=N}^\infty \left\|
 \mathcal{G}_\alpha^{-1} DR_\alpha(b)c \right\|_{(\ellnu)^n} 
 \leq B_g \cdot \sum_{\alpha=N}^\infty \left\|DR_\alpha(b)c \right\|_{(\ellnu)^n}.
\end{align*}
The remainder of the proof is a direct application of Theorem \ref{thm:contr}. 
Finally, we note that $a^\infty$ is the Fourier-Taylor expansion of the tail for the 
exact solution of \eqref{eq:InvarianceFourier}. This allows us to complete the 
estimate in equation \eqref{eq:BoundTailUnknown} and obtain
\begin{align}
\|P - P^N\|_\infty &\leq \sum_{\alpha=0}^{N-1}
 \|A_\alpha - \bar A_\alpha\|_\infty + \left\| \sum_{\alpha=N}^\infty  A_\alpha \right\| \\
&\leq r^N + \left\| \sum_{\alpha=N}^\infty  A_\alpha \right\| \\
&\leq r^N + r^\infty.
\end{align}
We return to the examples of the Lorenz and 
Hill systems to exhibit the construction of the bounds $Y$ and $Z(r)$
in practice.

\begin{example}[The $Y$ and $Z(r)$ bounds]
In the case of the Lorenz system, there are only two convolution products. 
Let $\bar a^N = (\bar a^1,\bar a^2, \bar a^3) \in \left( X_\nu^N\right)^3$, and $e=(e^1,e^2,e^3) \in \left( X_\nu^N\right)^3$ a point in the ball of radius $r^N$ about the numerical approximation. 
We have that
\begin{align*}
\sum_{\alpha = N}^\infty \left\| (a^1)^N \ast (a^3)^N \right\|_{1,\nu} &= \sum_{\alpha = N}^{2N} \left\| \left([\bar a^1 + e^1] \hat \ast [\bar a^3 +e^3] \right)_\alpha \right\|_{1,\nu} \\
&\leq \sum_{\alpha = N}^{2N} \left\| (\bar a^1 \hat\ast \bar a^3)_\alpha \right\|_{1,\nu} +\left\| (\bar a^1 \hat\ast \bar e^3)_\alpha \right\|_{1,\nu} +\left\| (\bar e^1 \hat\ast \bar a^3)_\alpha \right\|_{1,\nu} +\left\| (e^1 \hat\ast e^3)_\alpha \right\|_{1,\nu} \\
&\leq  \sum_{\alpha = N}^{2N} \left\| (\bar a^1 \hat\ast \bar a^3)_\alpha \right\|_{1,\nu} +\left\| (\bar a^1 \hat\ast \bar e^3)_\alpha \right\|_{1,\nu} +\left\| (\bar e^1 \hat\ast \bar a^3)_\alpha \right\|_{1,\nu} +\left\| (e^1 \hat\ast e^3)_\alpha \right\|_{1,\nu} \\
&\leq r^N\left( \left\| \bar a^1 \right\|_{X^N} +\left\| \bar a^3 \right\|_{X^N} +r^N\right)+ \sum_{\alpha = N}^{2N} \left\| (\bar a^1 \hat\ast \bar a^3)_\alpha \right\|_{1,\nu}, \\
\end{align*}
and note that the remaining finite sum is bound using interval arithmetics. 
The same computation for the convolution product in the third component
leads to 
\[
Y= \max \begin{pmatrix}
r^N\left( \left\| \bar a^1 \right\|_{X_\nu^N} +\left\| \bar a^3 \right\|_{X_\nu^N} +r^N\right)+\displaystyle \sum_{\alpha = N}^{2N} \left\| (\bar a^1 \hat\ast \bar a^3)_\alpha \right\|_{1,\nu}, \\
r^N\left( \left\| \bar a^1 \right\|_{X_\nu^N} +\left\| \bar a^2 \right\|_{X_\nu^N} +r^N\right)+\displaystyle \sum_{\alpha = N}^{2N} \left\| (\bar a^1 \hat\ast \bar a^2)_\alpha \right\|_{1,\nu},
\end{pmatrix}
\]
which satisfies the required inequality. 
Note that the result is that the nonlinearity is applied to the numerical data and 
that derivative of the nonlinearity, evaluated at the numerical data, 
''feels'' the perturbation.  This is a general fact which allows us to 
extend the bounds easily to the Hill and other polynomial problems.   

To compute the polynomial bound $Z(r)$, let $c \in X^\infty$ denote two
elements with norm one. We seek a uniform bound on 
$\left\| DR(x)c\right\|_{\left(X^\infty\right)^3}$, where $x$ is an element in the 
ball of radius $r$ about the solution. 
Again, we will study only one of the non-zero component, the other component can 
be bounded similarly. 

We have that
\begin{align*}
\left\| \left[DR(x)c\right]^2\right\|_{X_\nu^\infty} &= \sum_{\alpha=N}^\infty \left\| \left[((a^1)^N +x^1) \hat\ast c^3 \right]_\alpha +\left[c^1 \hat\ast ((a^3)^N +x^3) \right]_\alpha  \right\|_{1,\nu} \\
&\leq \sum_{\alpha=N}^\infty \left\| ((a^1)^N \hat\ast c^3 )_\alpha +(x^1 \hat\ast c^3 )_\alpha +(c^1 \hat\ast (a^3)^N)_\alpha +(c^1 \hat\ast x^3)_\alpha  \right\|_{1,\nu} \\
& \leq \left\| (a^1)^N  \right\|_{X_\nu^N} +\left\| (a^3)^N  \right\|_{X_\nu^N} +2r,
\end{align*}
so that
\[
Z(r)= \max \begin{pmatrix}
\left\| (a^1)^N  \right\|_{X_\nu^N} +\left\| (a^3)^N  \right\|_{X_\nu^N} \\
\left\| (a^1)^N  \right\|_{X_\nu^N} +\left\| (a^2)^N  \right\|_{X_\nu^N}
\end{pmatrix} +2r
\]
This calculation also generalizes to the case of Hill's three body problem.
\end{example}

\begin{figure}\label{figure:ABtoAABmanifolds}
\centering
\includegraphics[width=\textwidth]{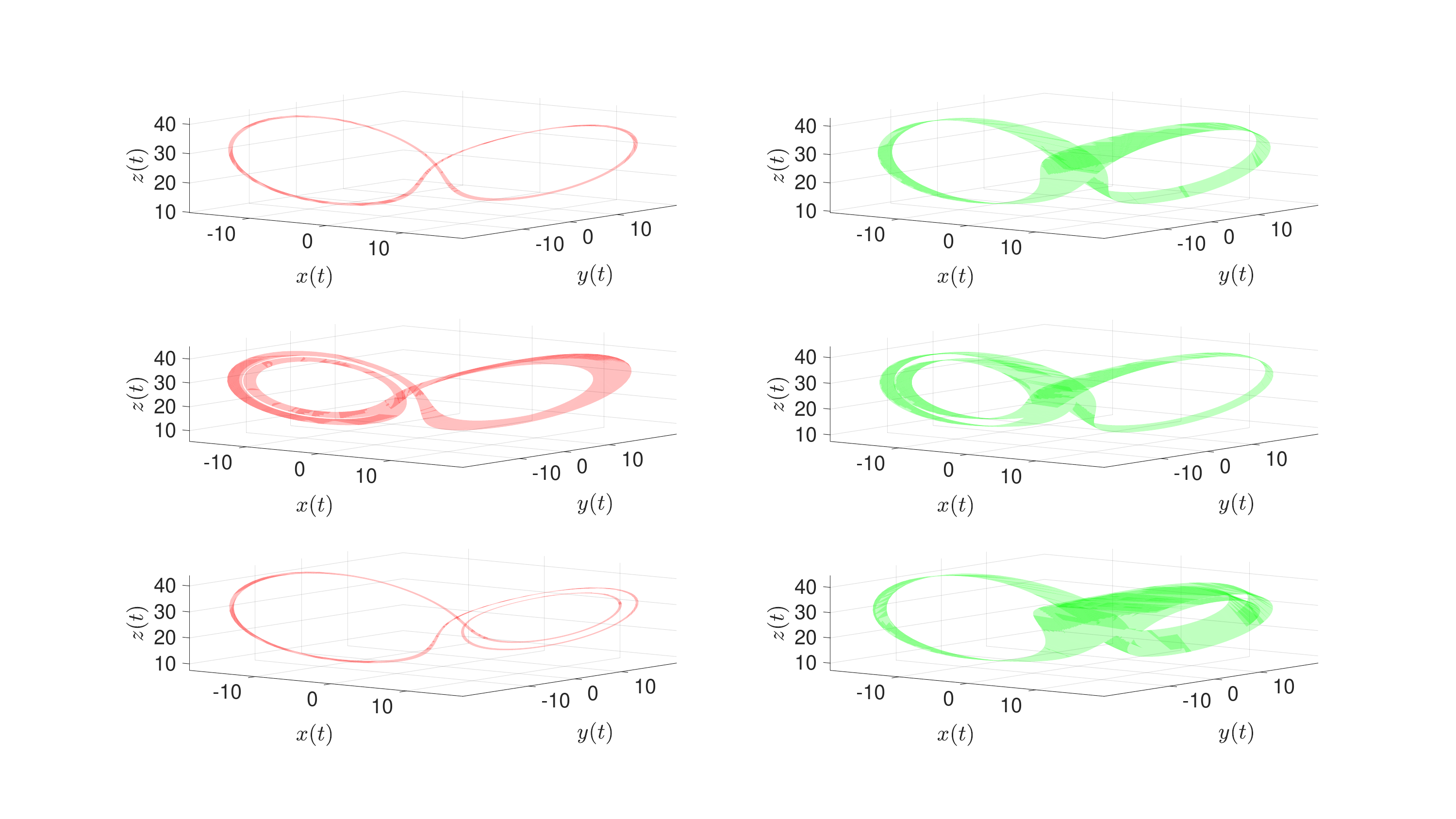}

\caption{Stable and unstable manifolds attached to the 
three shortest periodic solutions for the Lorenz equation at the classical parameters 
$\sigma=10, \rho=28, \beta=\frac{8}{3}$.  Recall that the periodic orbits were illustrated
in Figure \ref{fig:LorenzAttractor}.
The stable manifold attached to each of the three periodic solutions
are illustrated on the left (green), and the unstable 
manifolds are illustrated on the right (red). 
All parameterizations are computed with $k=75$ Fourier coefficients 
and $N=8$ Taylor coefficients. }
\end{figure}

\begin{figure}
\centering
\includegraphics[width=\textwidth]{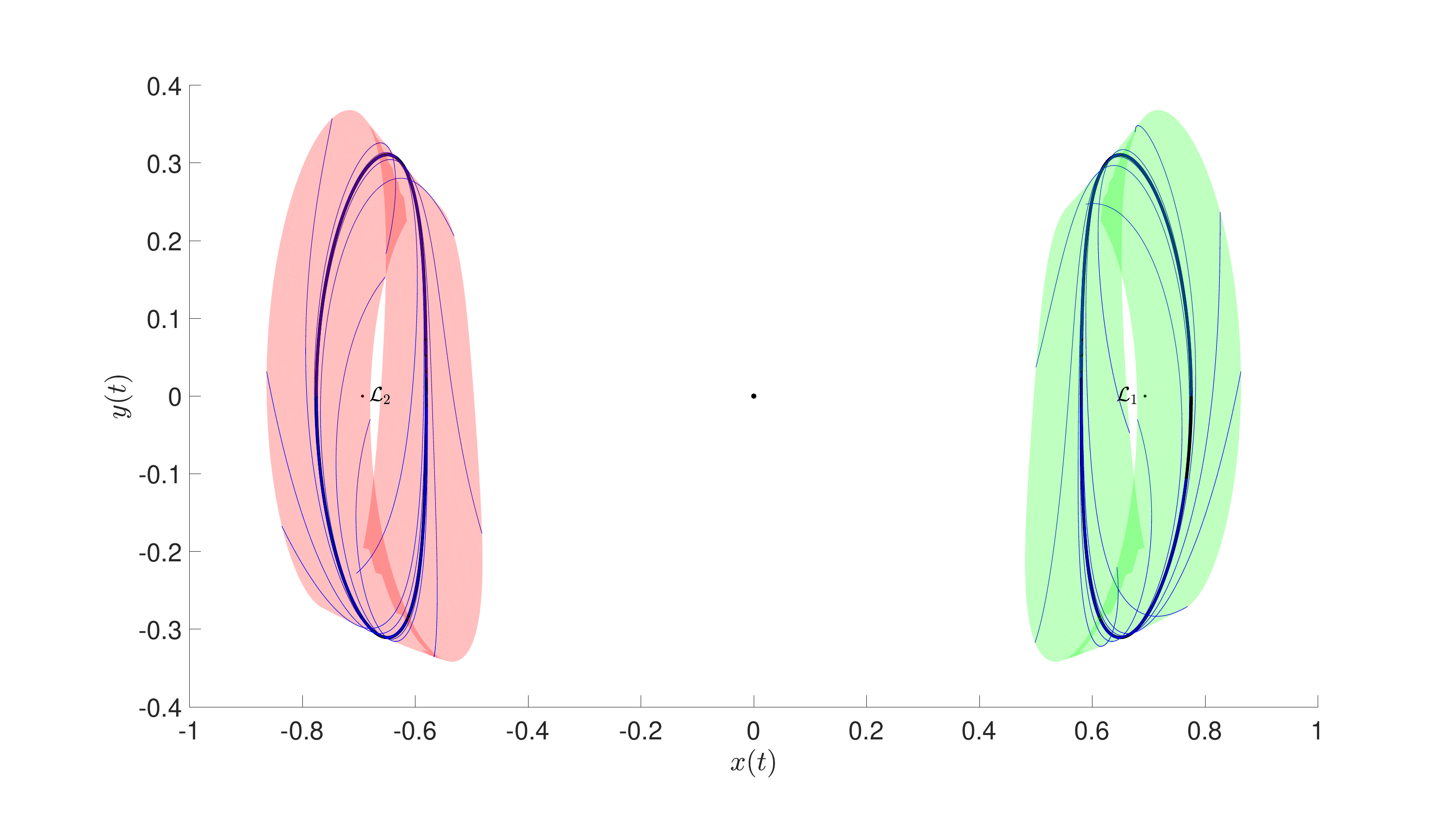}
\caption{A pair of periodic solutions to the HRFBP at $\mu= 0.00095$ and Jacobi constant of $C=4$. The Lyapunov orbit at $\mathcal{L}_1$ is displayed with a parameterization of its stable manifold (green) The Lyapunov orbit at $\mathcal{L}_2$ is displayed with its unstable manifold (red). Both manifold are computed with $N=8$ Taylor coefficients and $m=30$ Fourier coefficient. To illustrate the dynamic on each manifold, we apply the conjugacy relation \eqref{eq:parmMethod} to points evenly distributed on the boundary $|\sigma|=1$. The resulting trajectories are displayed in blue and accumulate, in forward time for the stable case and backward time for the unstable case, to the periodic orbit at the center of the cylinder.  } \label{fig:BVPschematic}
\end{figure}

%
%
%
%
%
%

%
\section{Connections between periodic orbits}\label{Section:Orbit}
Assume that for a given periodic solution, its
stable/unstable manifold is known using the approach described in 
Section \ref{Section:Manifold}.  Assume also that for each parameterized manifold 
we have the representations 
$P, \bar P:[0,T_1]\times [-1,1] \to \mathbb{R}^n$, 
$Q, \bar Q:[0,T_2]\times [-1,1] \to \mathbb{R}^n$, and $r^P,r^Q < \infty$ with 
\begin{enumerate}
\item The trajectory $\gamma_1(t)= P(t,0)$ and $\gamma_2(t)=Q(t,0)$ are periodic solutions of the ODE with periods $T_1$ and $T_2$ respectively.
\item For $\lambda_u$, the unstable Floquet multiplier of the periodic solution $\gamma_1$, then the trajectory $x_u(t)= P(t,x_0 e^{\lambda_u t} )$ is a solution of the ODE generated by $f$ for any choice of $x_0 \in [-1,1]$. Moreover, $x_u(t)$ accumulates to the periodic solution $\gamma_1$ in backward time.
\item For $\lambda_s$, the stable Floquet multiplier of the periodic solution $\gamma_2$, then the trajectory $x_s(t)= P(t,x_1 e^{\lambda_s t} )$ is a solution of the ODE generated by $f$ for any choice of $x_1 \in [-1,1]$. Moreover, $x_s(t)$ accumulates to the periodic solution $\gamma_2$ in forward time.
\item $\bar P, \bar Q$ are finite-dimensional Fourier-Taylor approximations of $P$ and $Q$ respectively, with 
\[
\| P - \bar P \|_\infty < r^{P},~\mbox{and}~ \| Q - \bar Q \|_\infty < r^{Q}.
\] 
We remark that the approximations do not need to have the same 
Fourier/Taylor dimension orders.
\end{enumerate}

If a trajectory $x(t)$, with the scalar unknowns 
$T>0$, $t_u \in [0,T_1]$, $t_s \in [0,T_2]$, 
and $ \sigma_u,\sigma_s \in [-1,1]$ solves 
the boundary value problem
\begin{align}\label{eq:BVPinitial}
\begin{cases}
x'(t)= f(x(t)), &\forall t \in (0,T), \\
x(0)= P(t_u,\sigma_u), & \\
x(T)= Q(t_s,\sigma_s), &
\end{cases}
\end{align} 
then $x(t)$ is a heteroclinic connecting orbit segment beginning on 
the unstable manifold of the periodic orbit $\gamma_1$ and
terminating on the stable manifold of the periodic orbit $\gamma_2$. 
Note that the flight time $T$ is finite, but unknown. 

Since the system of interest $f$ does not depend on time, 
a further rescaling of time allows us to remove the unknown 
from the definition of the domain of $x$. We set $L = \frac{T}{2}$, and 
seek to solve the problem
\begin{align}\label{eq:BVPChebyshevReady}
\begin{cases}
x'(t)= Lf(x(t)), &\forall t \in (-1,1), \\
x(-1)= P(t_u,\sigma_u), & \\
x(1)= Q(t_s,\sigma_s), &
\end{cases}.
\end{align} 
The problem \eqref{eq:BVPChebyshevReady} is equivalent to \eqref{eq:BVPinitial}, 
but with the addition that the unknown $x$ is such that
 $x: [-1,1] \to \mathbb{R}^n$. 


It follows from our previous assumptions about $f$ that the trajectory is real 
analytic, so that it can be expressed as a Chebyshev series whose 
coefficients decay exponentially. We recall that a real analytic 
function $h:[-1,1] \to \mathbb{R}$, which extends analytically to a 
Bernstein ellipse, can be uniquely expressed as 
\begin{align}\label{eq:ChebyExpansion}
	h(t)= y_0 + 2 \sum_{k=1}^\infty y_k T_k(t),
\end{align}
where $T_k(t)= \cos(k \arccos t )$ is the $k-$th Chebyshev polynomial. 
Chebyshev polynomials also satisfy the recurrence relation 
$T_{k+1}(t)= 2t T_k(t) -T_{k-1}(t)$ and have a well known antiderivative 
that allows a straighforward rewriting of the problem into the space of 
coefficients. It follows from the definition that we can set 
$y_{-k} = y_k$ to define a sequence of coefficients 
$y \in \ellnu$ for some $\nu > 1$ and show that
\[ 
	\| h \|_\infty \leq \|y \|_{1,\nu}.
\] 

Moreover, the product of two functions defined on $[-1,1]$ and expressed using 
Chebyshev expansion will be given by the discrete convolution product of their 
Chebyshev coefficient sequences. Hence, it will be possible to use the 
same approach as in the case of parameterized manifold to validate finite 
dimensional approximations. This approach and the rewriting of the problem into a zero 
finding operator is discussed in-depth in \cite{MR3148084}. We remark that 
solution arcs can be broken down into sub-arcs, all of which are computed using 
a corresponding Chebyshev expansion. This can be done to improve the accuracy 
of the interpolation and we refer to \cite{MR3754682} for a full discussion of this 
idea.  

In the present work, we assume that $M$ Chebyshev arcs are used to 
express the homoclinic orbit segment $x(t)$, 
solving \eqref{eq:BVPChebyshevReady}. 
That is, we take $\alpha_1,\alpha_2,\hdots,\alpha_M$, with 
\[
\sum_{i=1}^M \alpha_i = 1,
\]
to represent the flight time of the $i$-th arc as a proportion of the 
full flight time of the arc $x(t)$. For $1\leq i \leq M$, the half-frequency 
of the $i-$th Chebyshev arc is given by $L_i$ defined as
\[
L_i= \alpha_i L.
\] 
Each arc requires its own initial condition to determine the Chebyshev coefficients, 
and this condition will simply enforce the continuity of the solution arc $x(t)$.
That is, each arc's starting point in space will be the same point
as the previous arc's ending point. The first and last arcs are required to satisfy the 
appropriate boundary conditions from the initial problem \eqref{eq:BVPChebyshevReady}.
It follows that the $M$ arcs $x_i(t)$ interpolating $x(t)$ will satisfy the following $M$ boundary value problems
\begin{align*}
\begin{cases}
x_1'(t)= L_1 f(x(t)), &\forall t \in (-1,1), \\
x_1(-1)= P(t_u,\sigma_u), & \\
\end{cases},
\end{align*} 
in the first case, while the cases $2\leq i \leq M-1$ are given by
\begin{align*}
\begin{cases}
x_i'(t)= L_i f(x(t)), &\forall t \in (-1,1), \\
x_i(-1)= x_{i-1}(1), & \\
\end{cases},
\end{align*} 
and the last case $i=M$ is
\begin{align*}
\begin{cases}
x_M'(t)= L_m f(x(t)), &\forall t \in (-1,1), \\
x_M(-1)= x_{M-1}(1), & \\
x_M(1)= Q(t_s,\sigma_s). &
\end{cases}
\end{align*} 
More details are given in Remark \ref{rem:ChebyshevDecomposition}.

For $1\leq i \leq M$, let $y^i \in (\ellnu)^n$ represent the Chebyshev expansion of $x_i$, 
so that $y_{-k}^i= y_k^i$ for all $k\geq 0$ and $y= (y^1,y^2,\hdots,y^M) \in (\ellnu)^{Mn}$ 
represents the sequence of Chebyshev parameterization representing $x(t)$, 
the solution of the global BVP \eqref{eq:BVPChebyshevReady}. 
We recast the sequence of BVP into a problem in the space of Chebyshev coefficients. 
So that, for $1\leq i \leq M$, they satisfy $G^i(y)=0$ for some $G^i: (\ellnu)^n \to (\ellnu)^n$. 
Each $G^i$ is given by

%
%
%
%
\begin{align*}
\left(G^i(y)  \right)_k = \begin{cases}
\displaystyle \left( y_0^i + 2 \sum_{j=1}^\infty y_j^i (-1)^j \right) -x_0^i(y), & \mbox{for}~ k=0, \\
2k y_k^{i} + L_i \Lambda_k \left(f(y^i)\right), & \mbox{for}~ k \geq 1,
\end{cases}
\end{align*}
where $\Lambda_k$ is the linear operator $\Lambda_k: \ellnu \to \ellnu$ such that
\[
 \Lambda_k \left(f(y^i)\right) =  \left(f(y^i)\right)_{k+1} -  \left(f(y^i)\right)_{k_1}, ~\forall k \geq 1,
\] 
the initial condition $x_0^i$ is
\begin{align*}
x_0^i(y) = \begin{cases}
P(t_u,\sigma_u), & \mbox{if}~ i=1, \\
y_0^{i-1} + 2\displaystyle \sum_{j=1}^\infty y_j^{i-1} (-1)^j , & \mbox{if}~ i>1, \\
\end{cases}
\end{align*}
and $f(y^j)$ is the evaluation of the vector field at the Chebyshev expansion $y^j$. 
This expression is identical to the Fourier case presented in 
Examples \ref{Ex:Orbit} and \ref{Ex:Bundle}. We note that each problem 
uses the initial condition in its definition but the final condition $y^M(1)=Q(t_s,\sigma_s)$ 
is unused. This condition can be written as a solution to 
\[
\eta(y,t_s)= \displaystyle \left( y_0^M + 2 \sum_{j=1}^\infty y_j^M (-1)^j \right) -Q(t_s,\sigma_s) =0,
\]
and this scalar equation balances the scalar variable corresponding to the flight time, 
and the evaluation of both parameterization. The following example shows the operator
 and an approximate solution in the case of the two problems of interest.

\begin{remark}[Transversality] \label{rem:transverse}
{\em
Our computer assisted analysis of  Equation \eqref{eq:BVPChebyshevReady}
uses a Newton-Kantorovich argument to establish the existence of a 
true zero of the system of equations in a small neighborhood of a 
good enough numerically computed approximate solution.
As a byproduct, it is a standard consequence of the Newton-Kantorovich
argument that this zero is also non-degenerate.  
Transversality of the intersection now follows 
from non-degeneracy exactly as in Theorem 7 of 
\cite{MR3207723}.  See also \cite{MR2821596}.
}
\end{remark}

\begin{example}[Connecting orbits in Lorenz]\label{Ex:ConnectingLorenz}
We present the operator $G^i$ in the case $i=1$, so that the initial condition 
requires the evaluation of the unstable parameterized manifold. $G^1$ will have 
three component $G^{1,1}$, $G^{1,2}$, and $G^{1,3}$ given by
\begin{align*}
G^{1,1}_k(L,\theta_u,\sigma_u, y) &= \begin{cases}
\left( y_0^{1,1} + 2 \displaystyle \sum_{j=1}^\infty y_j^{1,1}(-1)^j \right) - P^1(\theta_u,\sigma_u), & \mbox{if}\quad k=0, \\
2k y_k^{1,1} +\alpha_1 L\Lambda_k \left( \sigma( y^{1,2}-y^{1,1}) \right), & \mbox{if}\quad k>0, \\
\end{cases} \\
G^{1,2}_k(L,\theta_u,\sigma_u, y) &= \begin{cases}
y_0^{1,2} + 2 \displaystyle \sum_{j=1}^\infty y_j^{1,2} - P^{2}(\theta_u,\sigma_u), & \mbox{if}\quad k=0, \\
2k y_k^{1,2} +\alpha_1 L \Lambda_k \left(\rho y^{1,1}  -y^{1,2}  -(y^{1,1}\ast y^{1,3}) \right), & \mbox{if}\quad k>0, \\
\end{cases} \\
G^{1,3}_k(L,\theta_u,\sigma_u, y) &= \begin{cases}
y_0^{1,3} + 2 \displaystyle \sum_{j=1}^\infty y_j^{1,3} - P^{3}(\theta_s,\sigma_s), & \mbox{if}\quad k=0, \\
2k y_k^{1,2} +\alpha_1 L \Lambda_k \left( -\beta y^{1,3} + (y^{1,1}\ast y^{1,2}) \right), & \mbox{if} \quad k>0. \\
\end{cases} 
\end{align*}
The scalars $\theta_s,\sigma_s$ are also variable in this problem, but they are omitted from the 
left hand side of the previous equations, as they are only involved in the definition of the 
scalar condition $\eta$. Define $G:\mathbb{R}^3 \times \left( \ellnu \right)^{3M} \to \mathbb{R}^3 \times \left( \ell_{\nu'}^1 \right)^{3M} $ for some $\nu' < \nu$ as
\begin{align}
G(L,\theta_u,\theta_s,y) = \left( \eta(L,\theta_s,y), G^1(L,\theta_u,y), G^2(L,y), \hdots, G^M(L,y) \right). \label{eq:ChebyshevProblemLorenz}
\end{align}
\end{example}

\begin{remark}[On the domain decomposition of the Chebyshev Arc]\label{rem:ChebyshevDecomposition}
{\em
Note that the choice of the partition values $\alpha_i$ will affect the length of the 
respective arc, and therefore the decay of its Chebyshev coefficients. This choice 
does not introduce new variables to the system, as the partition is fixed and the global 
integration time $L$ remains the only unknown. This choice is made by considering 
the properties of the numerical solution, before 
any attempt at validation. This allows to take Chebyshev sequence of same dimension for all $M$ arcs, thus simplifying the writing of the numerical computation and its validation. The work of
\cite{MR4292534} provides a scheme to generate the optimal partition,
leading to better numerical stability. Moreover, 
the number of subdomains $M$ is chosen so that the last coefficients for 
each truncated Chebyshev expansion are below a chosen threshold 
(near machine precision). }
\end{remark}

\begin{remark}[On the choice of scalar variables]
{\em
In equation \eqref{eq:ChebyshevProblemLorenz}, we removed $\sigma_u$ and $\sigma_s$ from the variables of the problem in order to keep 
the system balanced. This choice also guarantees that each parameterization will be
 evaluated within their respective domain of definition, as both parameterized manifolds are 
 periodic in the remaining variables. Note however that it is also possible to fix $t_u$ and $t_s$, 
 but that this would require extra attention to ensure that the solution does not fall outside the 
 domain of $P$ and $Q$. This scenario might be beneficial to use despite the added constraint, 
 as it sometimes leads to a better conditioned numerical inverse.}
\end{remark}

 The set of variables will be
\begin{equation*}
	(L,\theta_u,\theta_s,y^{(1,1)},\hdots,y^{(M,n)}) \in \mathbb{R}^3 \times \left( \ellnu \right)^{Mn}  \bydef Y.
\end{equation*}
The space $Y$ just defined is a Banach space with norm 
\[
\|y \|_Y = \max\lbrace |L|,|\theta_u|, |\theta_s| ,\|y^{(1,1)}\|_{1,\nu} , \hdots, \|y^{(T,3)}\|_{1,\nu}  \rbrace.
\] 
An approximated solution is again validated using Theorem \ref{Radii}. 
The computation of the $Y$ and $Z(r)$ bounds are 
similar to the Fourier-Taylor case discussed above, and are
 omitted from this work. The bounds are however quite similar to the previous 
 cases treated in \cite{PaperBridge}.

%
%

\section{Examples} \label{sec:examples}

\subsection{Cycle-to-cycle connections in the Lorenz equations} \label{sec:lorenzResults}
Recalling the convention that the letters $A$ and $B$ denote
the turning of a periodic orbit in the Lorenz system about the 
left and right ``eyes'' respectively, we use the parameterization 
method described above to prove the following.  

\begin{theorem}
For the periodic orbit $\gamma_{AB}$ (resp.  $\gamma_{AAB}$, or $\gamma_{ABB}$), 
there exists a pair of two-dimensional stable and unstable manifold 
$P_{AB}$, $Q_{AB}$ (resp. $P_{AAB}$, $Q_{AAB}$ or $P_{ABB}$, $Q_{ABB}$), 
solution to \eqref{eq:parmMethod}. The finite dimensional approximation satisfies
\begin{align*}
\left\| P_{AB} - \bar P_{AB} \right\|_\infty < 1.2027 \cdot 10^{-11},
 \quad &\left\| Q_{AB} - \bar Q_{AB} \right\|_\infty < 2.9500 \cdot 10^{-10}, \\
\left\| P_{AAB} - \bar P_{AAB} \right\|_\infty < 1.1685 \cdot 10^{-10}, 
\quad &\left\| Q_{AAB} - \bar Q_{AAB} \right\|_\infty < 4.0663 \cdot 10^{-10}, \\
\left\| P_{ABB} - \bar P_{ABB} \right\|_\infty < 5.9731 \cdot 10^{-10}, 
\quad &\left\| Q_{ABB} - \bar Q_{ABB} \right\|_\infty < 1.0012 \cdot 10^{-9}, \\
\end{align*}
Each computation is performed with $m=70$ Fourier modes, 
$N=8$ Taylor modes, and $\nu=1.03$.
\end{theorem}

Having in hand the periodic orbits, their bundles, and their invariant manifolds, 
we compute heteroclinic orbits as solutions of Equations \eqref{eq:BVPinitial}. 
There are six possible heteroclinic pairs which we labeled $I$ to $VI$ using the ordering;
\begin{enumerate}
 \item[I]: $P_{AB}$ to $Q_{AAB}$,
 \item[II]: $P_{AB}$ to $Q_{ABB}$, 
 \item[III]: $P_{AAB}$ to $Q_{AB}$, 
 \item[IV]: $P_{AAB}$ to $Q_{ABB}$, 
 \item[V]: $P_{ABB}$ to $Q_{AB}$, and 
 \item[VI]: $P_{ABB}$ to $Q_{AAB}$.
\end{enumerate}

\begin{theorem}
Let $x_i$ denote the numerically computed heteroclinic
orbits segments illustrated in Figure \ref{fig:lorenzConnections}
for the six connecting orbit problems above. Then, the
true solutions of \eqref{eq:BVPinitial} have 
\begin{align*}
\left\| x_{I} - \bar x_{I} \right\|_{\infty} < 4.0072 \cdot 10^{-8}, \\
\left\| x_{II} - \bar x_{II} \right\|_{\infty} < 6.2914 \cdot 10^{-8}, \\
\left\| x_{III} - \bar x_{III} \right\|_{\infty} < 1.5951 \cdot 10^{-8}, \\
\left\| x_{IV} - \bar x_{IV} \right\|_{\infty} < 5.8614 \cdot 10^{-8}, \\
\left\| x_{V} - \bar x_{V} \right\|_{\infty} < 5.8690 \cdot 10^{-8}, \\
\left\| x_{VI} - \bar x_{VI} \right\|_{\infty} < 4.7930 \cdot 10^{-8}. \\
\end{align*}
Each connecting orbit segment is computed using
$4$ Chebyshev arcs and a uniform meshing of time. 
Each Chebyshev step is approximated using $100$ modes. 
\end{theorem}

\begin{center}
\begin{figure}
\subfigure{{\includegraphics[width=0.5\textwidth]{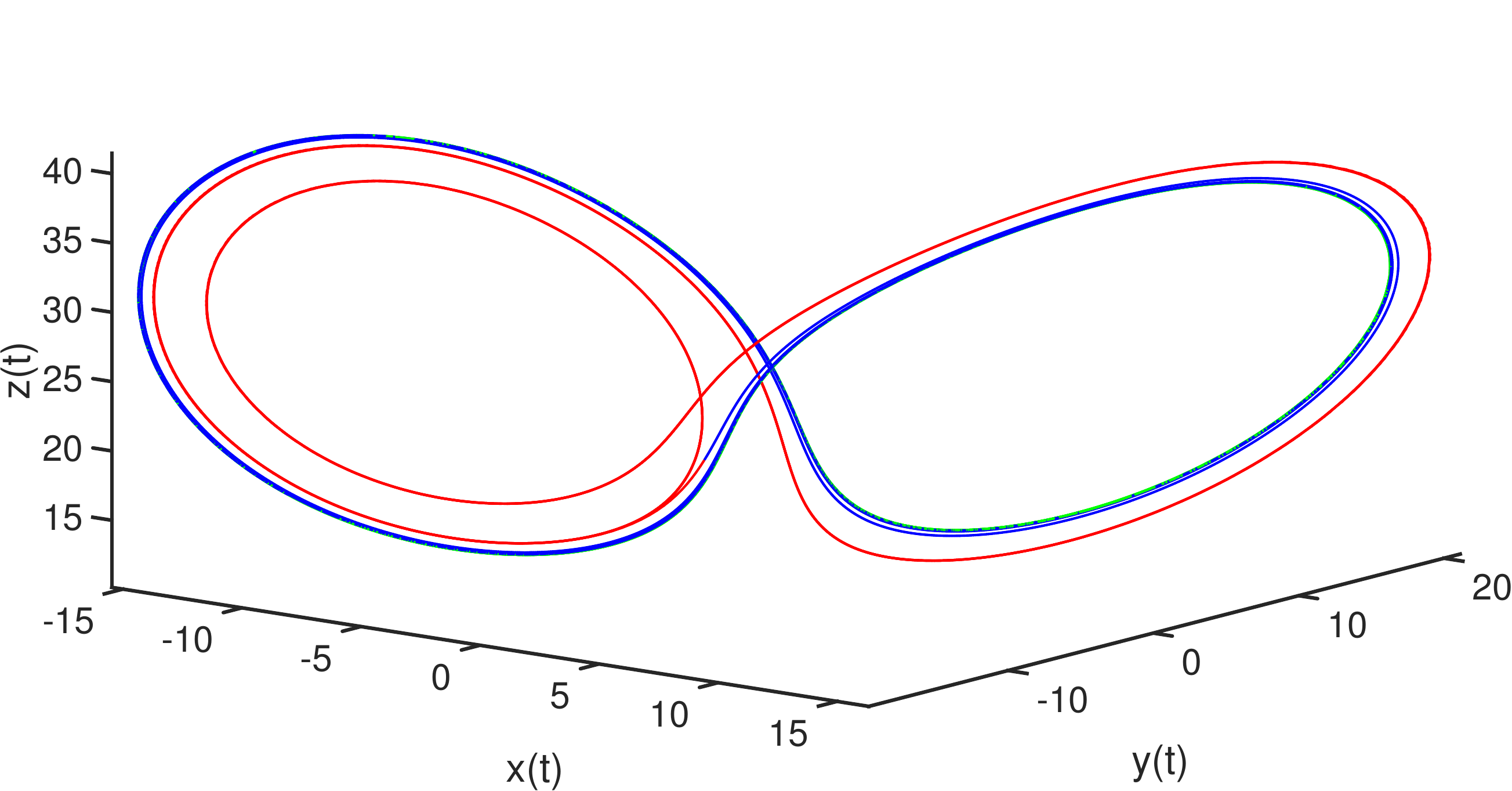}}}
\subfigure{{\includegraphics[width=0.5\textwidth]{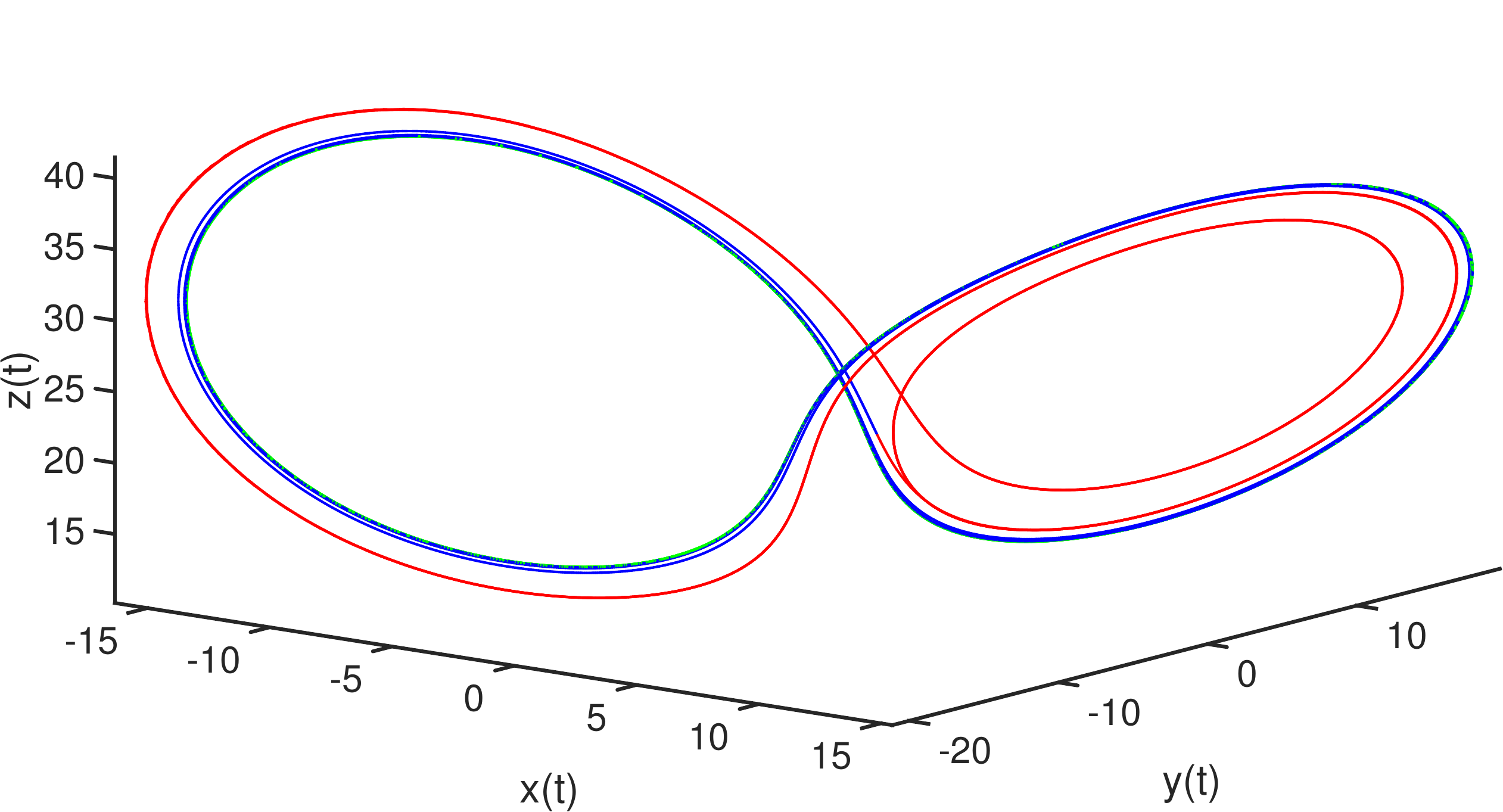}}} \\
\subfigure{{\includegraphics[width=0.5\textwidth]{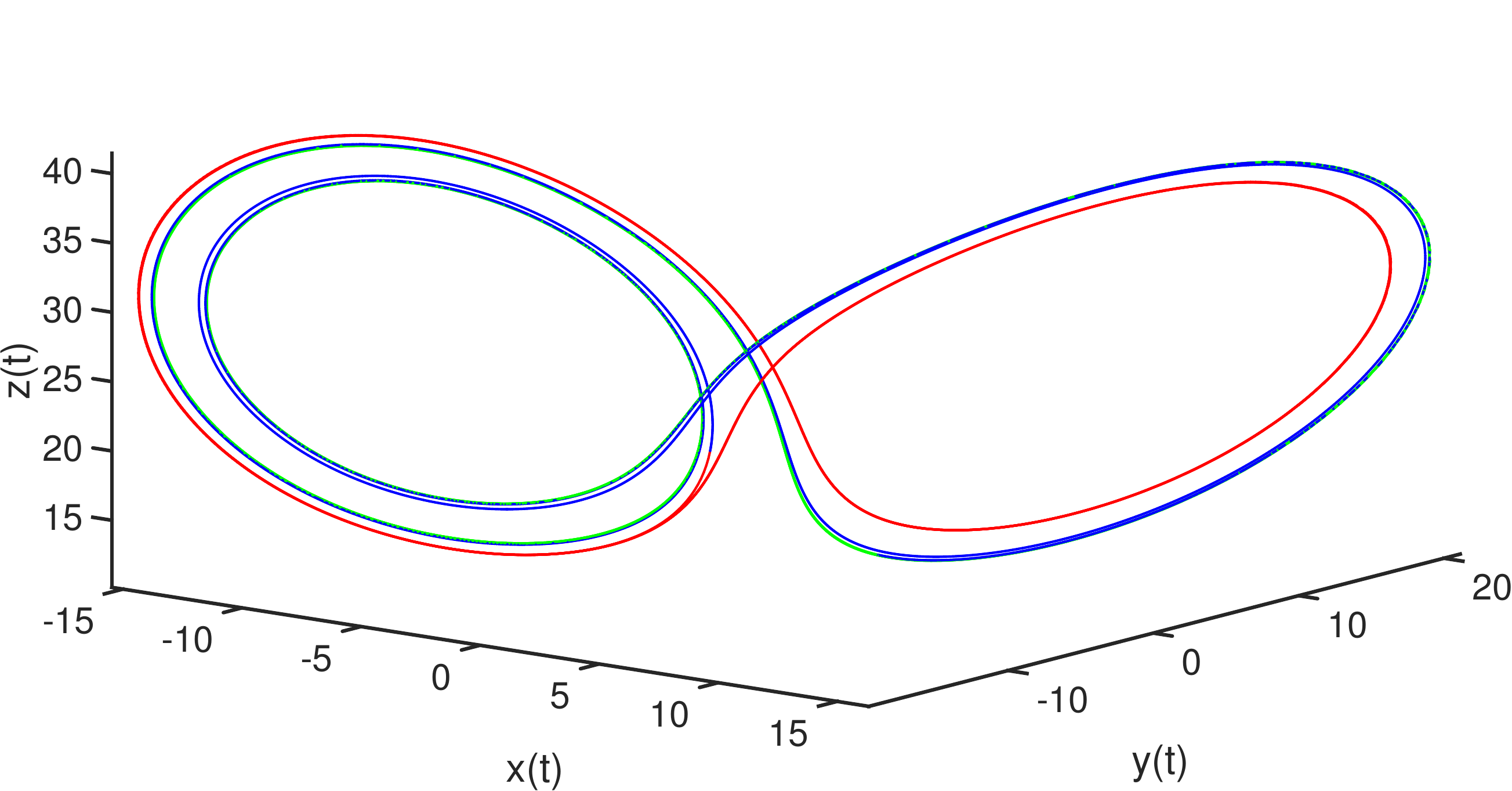}}}
\subfigure{{\includegraphics[width=0.5\textwidth]{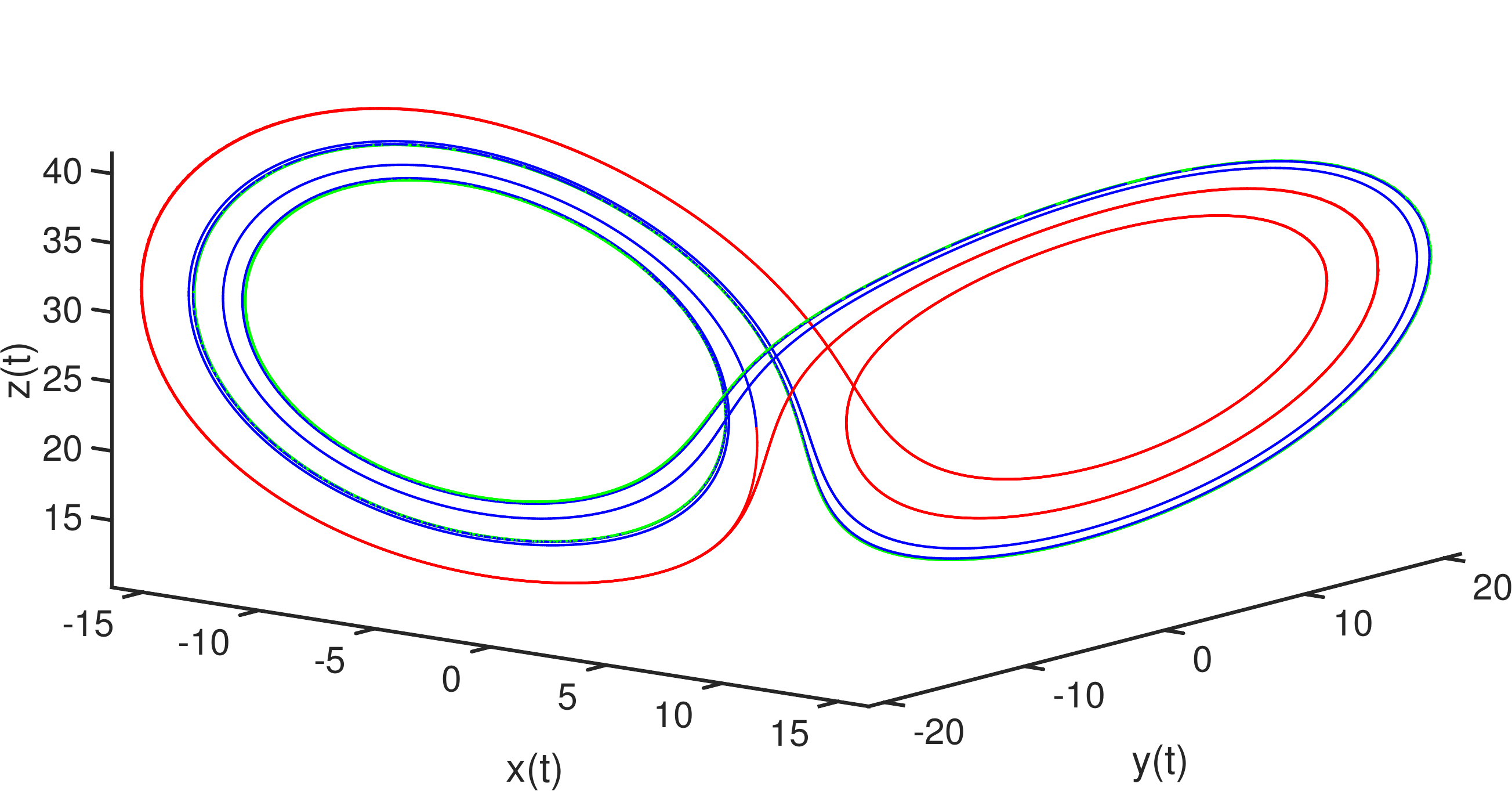}}} \\
\subfigure{{\includegraphics[width=0.5\textwidth]{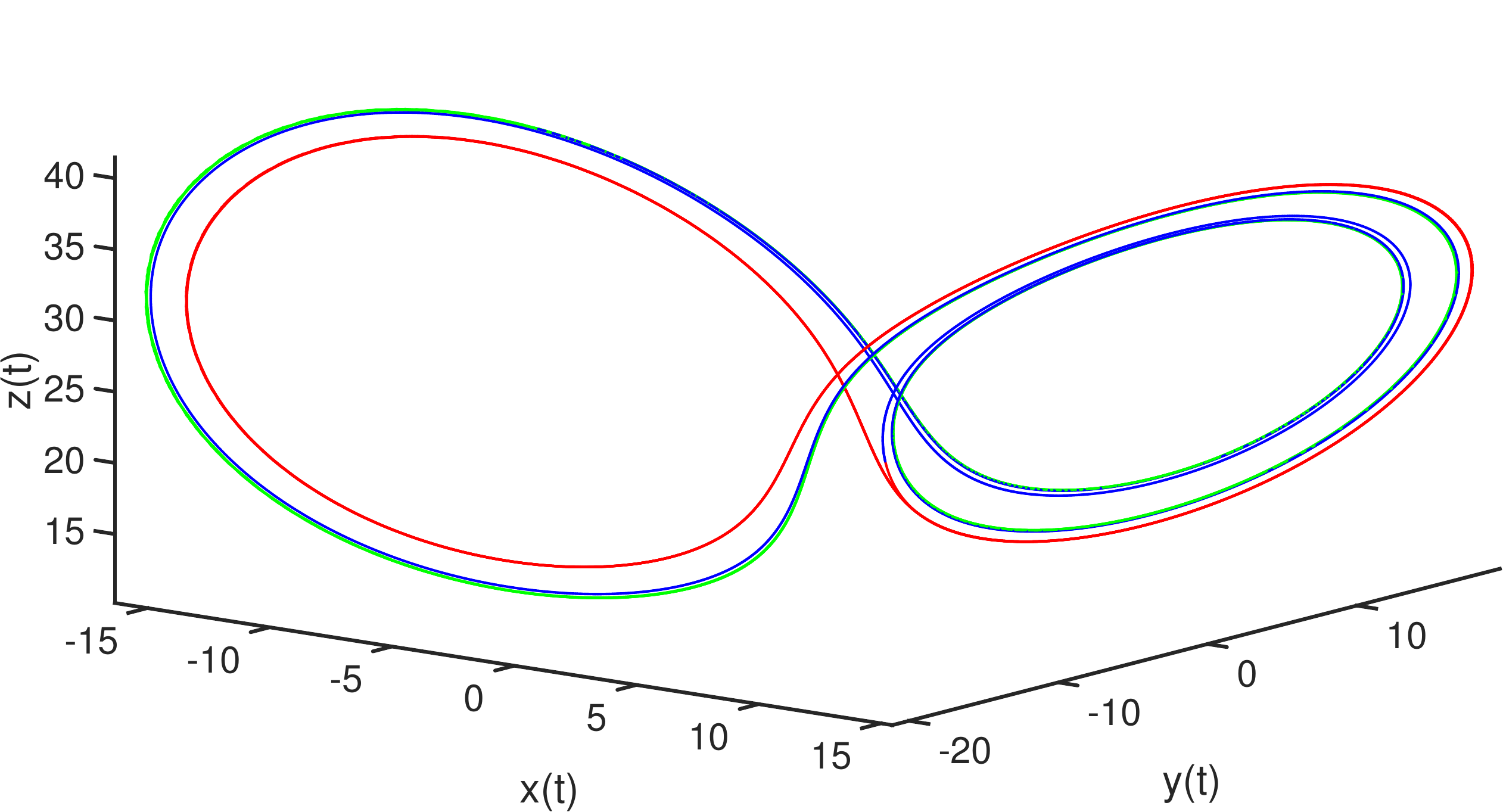}}}
\subfigure{{\includegraphics[width=0.5\textwidth]{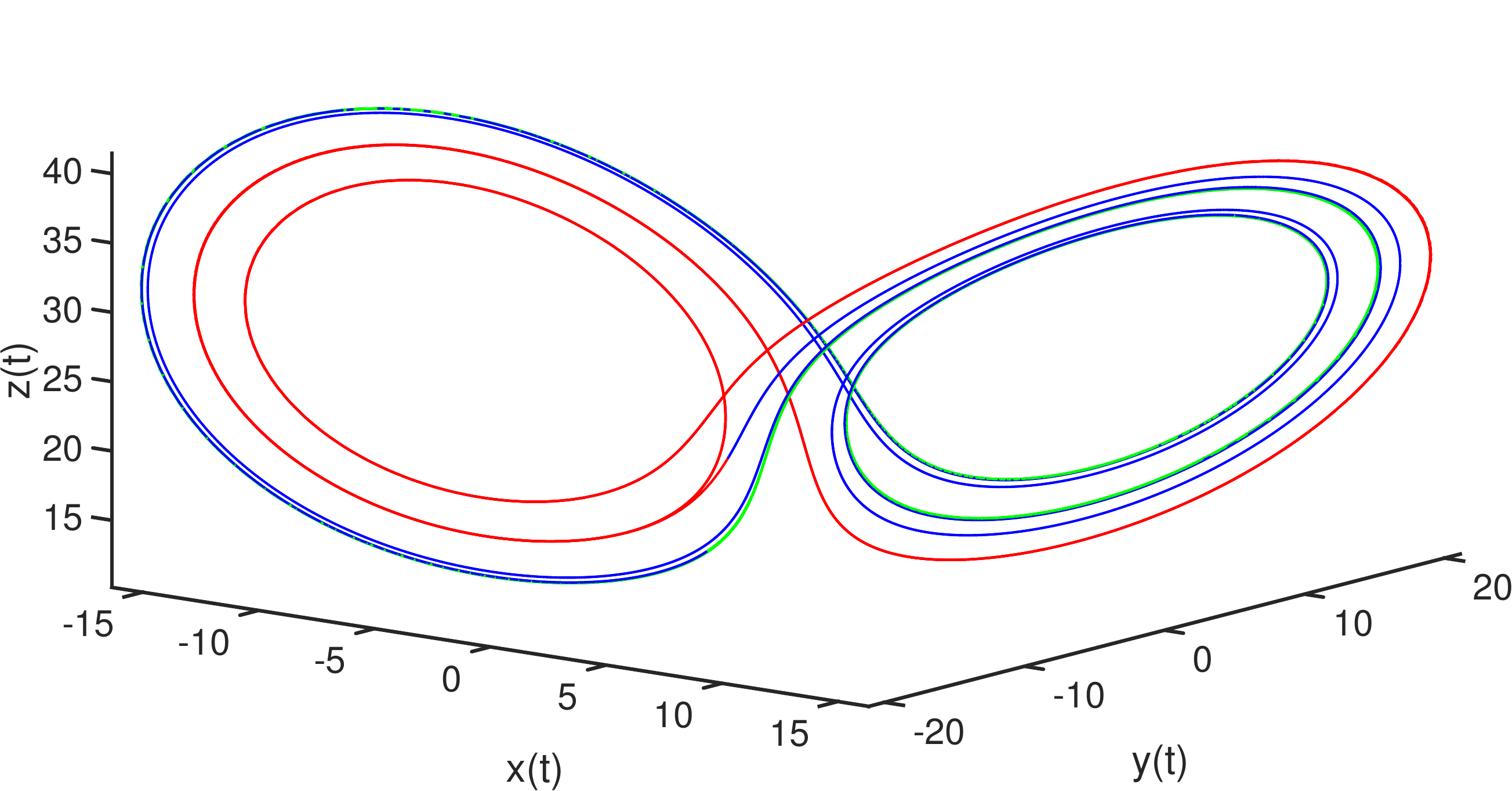}}}
\caption{ Six solutions to \eqref{eq:BVPinitial}. The first row represents the connecting orbit $I$ (left) and $II$ (right). Both connecting orbits accumulate to the periodic solution $AB$ at $-\infty$, connection $I$ accumulates to $AAB$ at $+\infty$ while $II$ accumulates to $ABB$. The Chebyshev arcs are represented in blue, while the arcs in red and green are integration-free (that is, these portions of the orbit are on the parameterized 
stable/unstable manifolds). and are calculated via the appropriate conjugacy relation. 
The second row represents the connecting orbit $III$ (left, between $AAB$ and $AB$) 
and $IV$ (right, between $AAB$ and $ABB$). The third row represents the connecting orbit
 $V$ (left, between $ABB$ and $AB$) and $VI$ (right, between $ABB$ and $AAB$).     
 } \label{fig:lorenzConnections}
\end{figure}
\end{center}

\begin{center}
\begin{figure}
\subfigure{{\includegraphics[width=0.5\textwidth]{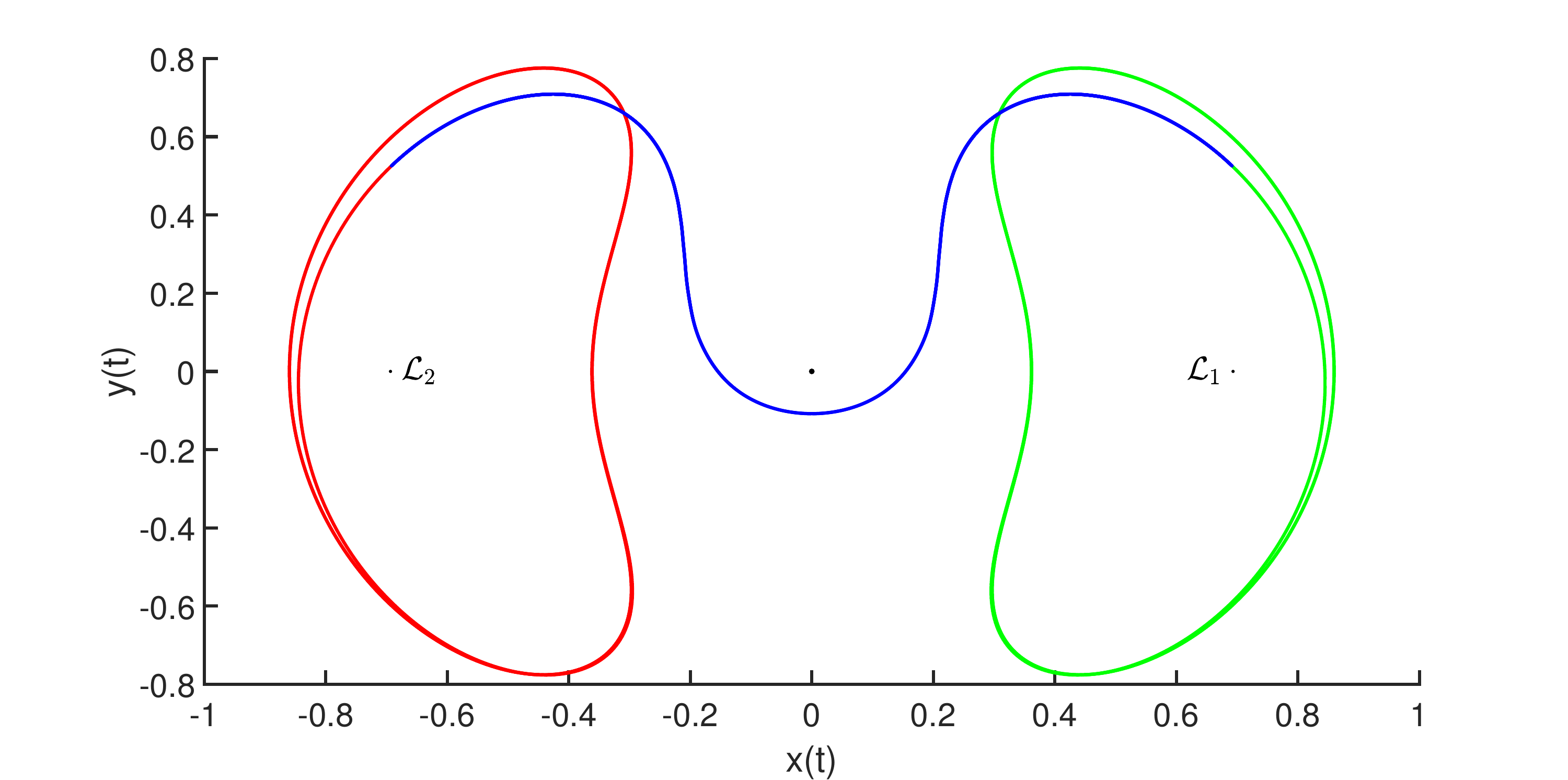}}}
\subfigure{{\includegraphics[width=0.5\textwidth]{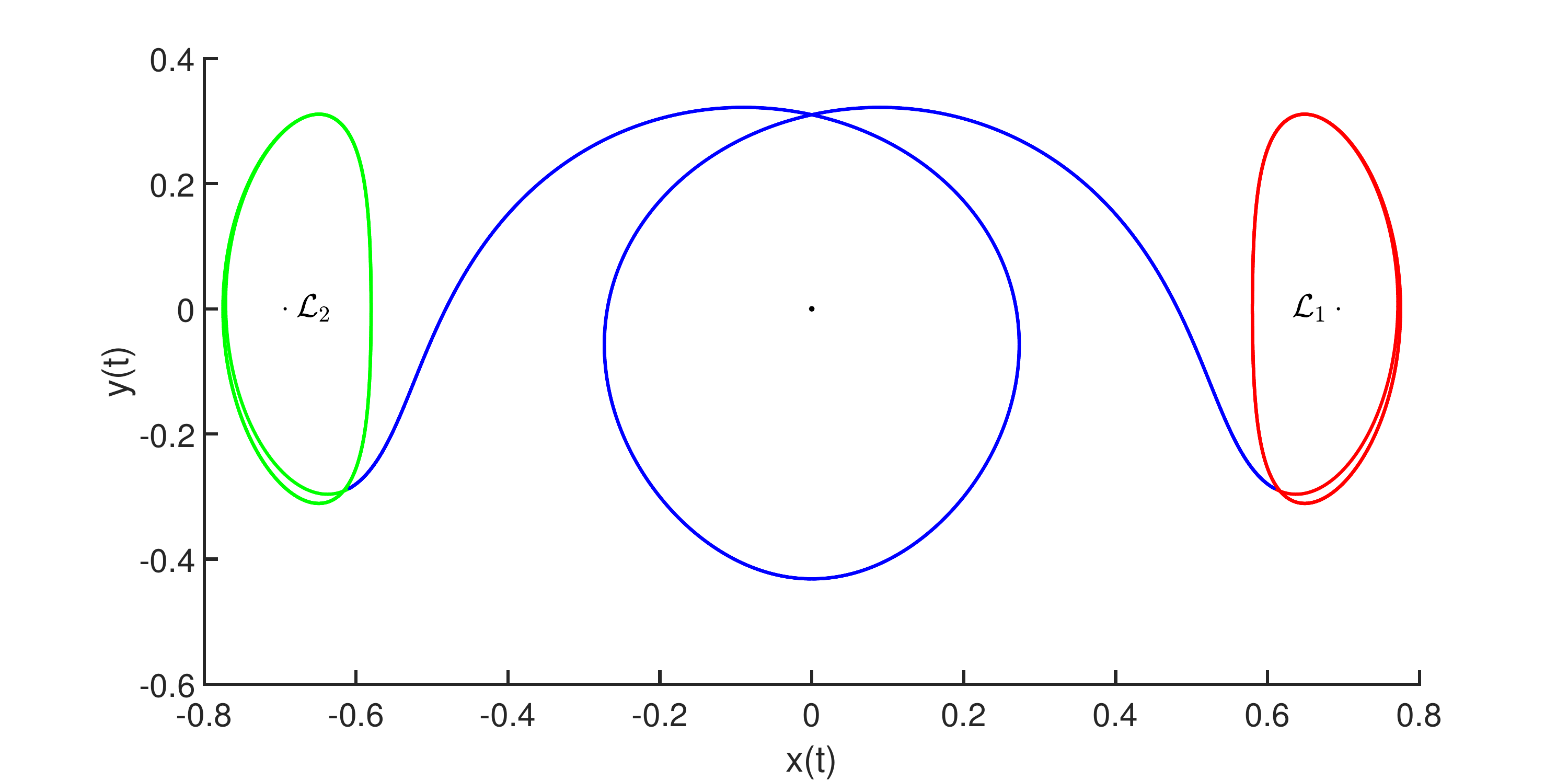}}}
\caption{ Heteroclinic connections between $\mathcal{L}_1$ and
 $\mathcal{L}_2$ planar Lyapunov families. 
 For both figures, the solution to the BVP is interpolated using 
 $3$ Chebyshev subdomains (blue) while green/red
 depicts the orbit segments on the
 the parameterized stable/unstable manifold. 
The connection on the left is for fixed energy level of $C = 2.5$ and has
no winding about the origin. The right frame illustrates 
a connection in the $C = 4$ energy level
with one winding about the origin. 
Both of these solutions are in the case $\mu=0$, 
that is, the classical Hill's lunar problem.  Then, due to the reversible 
symmetry of Hill's problem, there are symmetric connections running 
the other way (flipped about the x-axis) and hence a transverse heteroclinic 
cycle, implying the existence of Smale horse shoes. } \label{figure:HillsMu0}
\end{figure}
\end{center}

\begin{center}
\begin{figure}
\subfigure{{\includegraphics[width=0.48\textwidth]{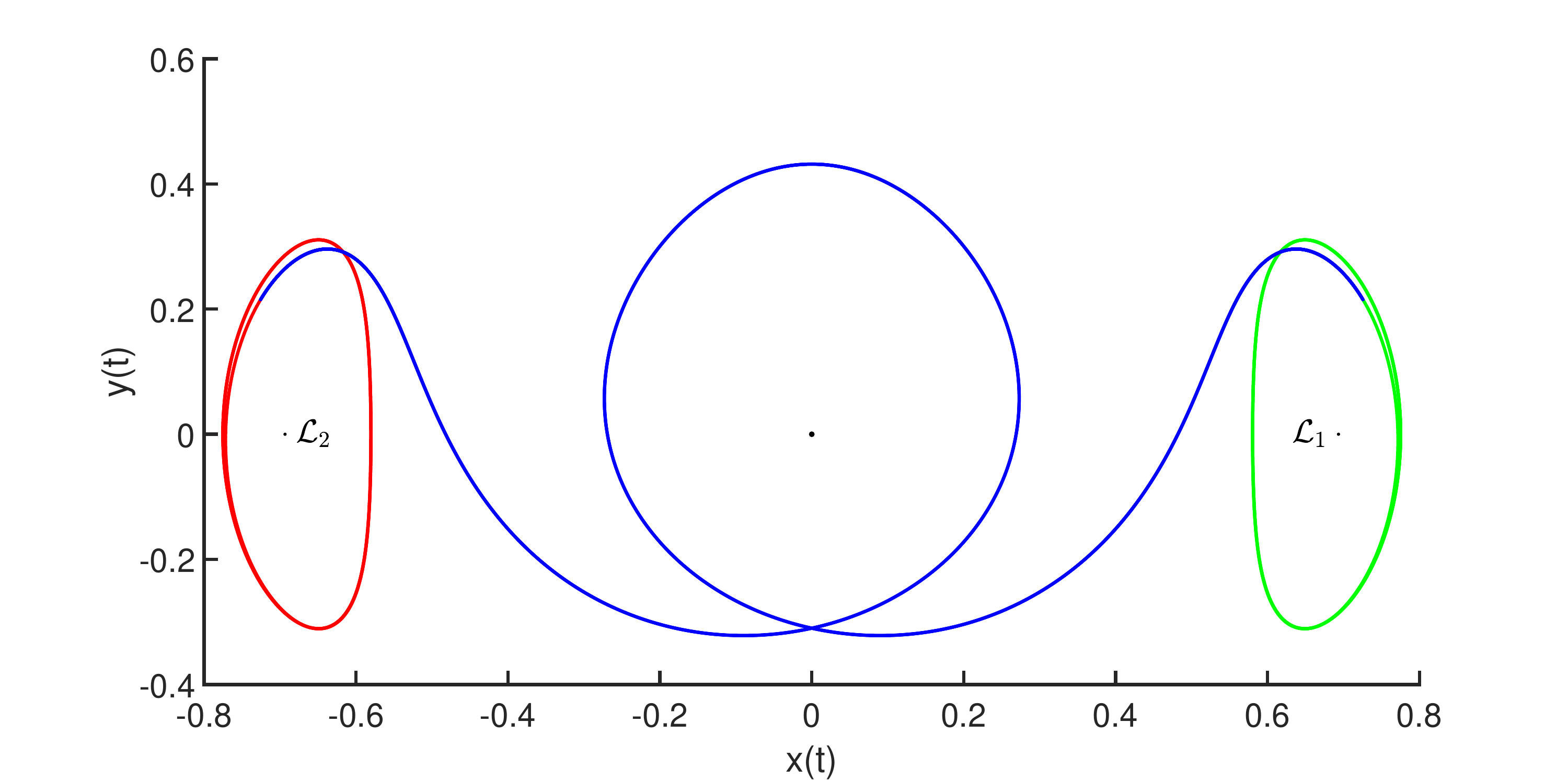}}}
\subfigure{{\includegraphics[width=0.48\textwidth]{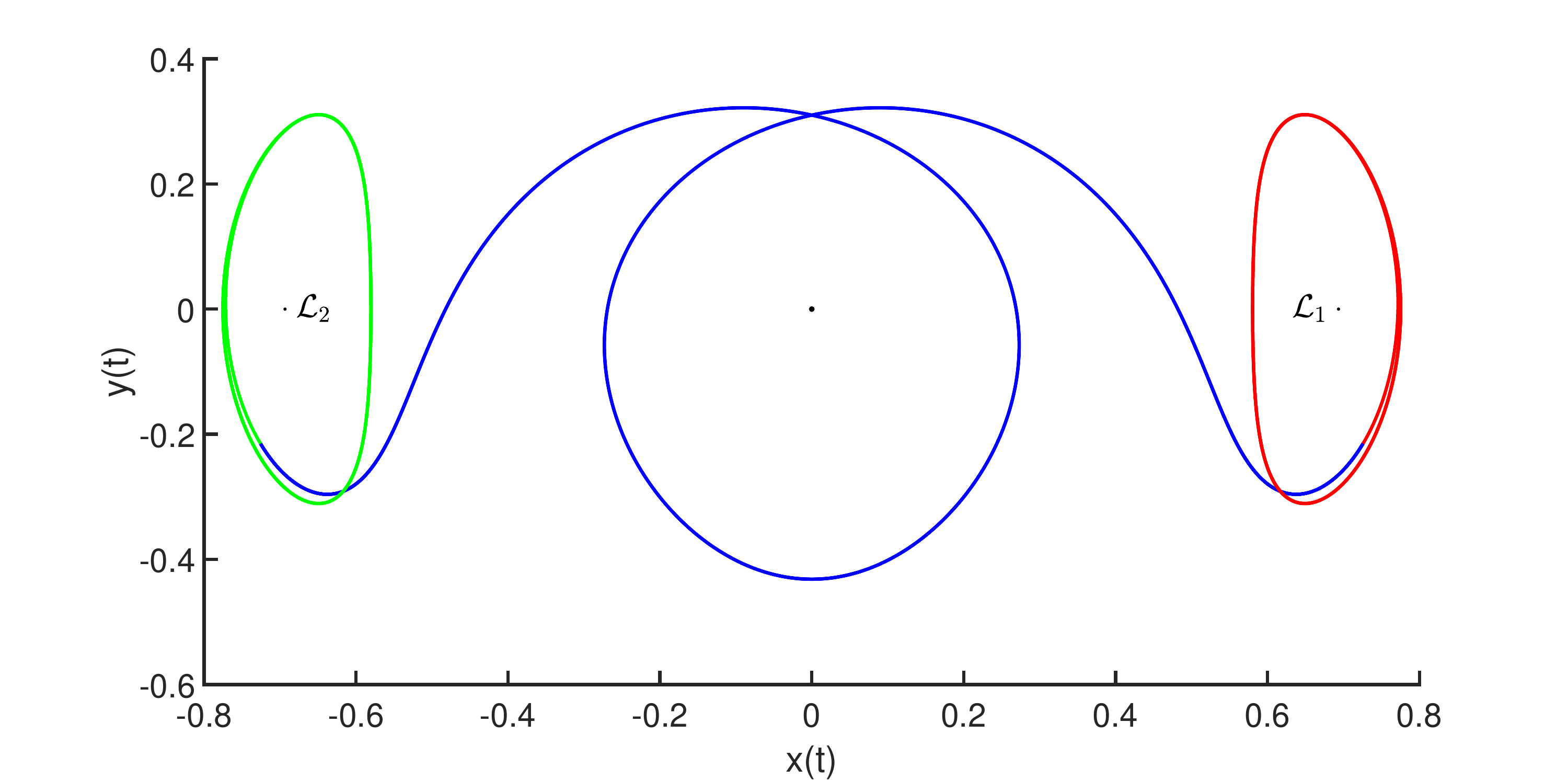}}}
\caption{ Heteroclinic connections between 
$\mathcal{L}_1$ and $\mathcal{L}_2$ planar Lyapunov orbits
at energy level of $C=4$ in the case $\mu=0.00095$. This setting corresponds to the 
mass ratio of the Jupiter-Sun system. We note that the periodic orbits and the pair 
of connections joining 
them are no longer symmetric with respect to the $y-$axis (though this 
asymmetry is difficult to detect visually).  For this reason we 
establish the existence of transverse connections in both directions (left and right 
frames).  Again, the heteroclinic cycle implies the existence of chaotic dynamics
in the energy level.} \label{figure:HillsMuJove}
\end{figure}
\end{center}

\subsection{Cycle-to-cycle connections in Hill's lunar problem} \label{sec:HillResults}
The following theorem, established using the techniques described above, 
gives the existence of a transverse heteroclinic cycle in 
Hill's Lunar problem (Hill restricted four body problem with $\mu = 0$)
with equal masses, when the Jacobi 
constant is $2.5$ or $4$.

\begin{theorem}\label{thm:ValidationE=4}
The lunar Hill's problem
admits transverse heteroclinic cycles between the $\mathcal{L}_1$
and $\mathcal{L}_2$ planar Lyapunov families for both $C = 2.5$ and 
$C = 4$.  
\end{theorem}

Note that there is nothing special about the energy levels $C= 2.5$ and $C= 4$.
Many similar theorems could be proven using the same methods.  

For this theorem, we prove first the 
existence a pair of periodic orbits in the $\mathcal{L}_1$ and $\mathcal{L}_2$ 
Lyapunov families for both $C = 2.5$ and $C=4$. The unstable manifold 
$P_{\mathcal{L}_2}$, solving Equation \eqref{eq:parmMethod}
is approximated by $\bar P_{\mathcal{L}_2}$ with $30$ Fourier modes and $4$ Taylor modes,
so that
\[
\left\| P_{\mathcal{L}_2} -\bar P_{\mathcal{L}_2} \right\|_\infty \leq 1.5175\cdot 10^{-11}.
\] 
Similarly, the stable manifold $Q_{\mathcal{L}_1}$, attached to the periodic solution
around $\mathcal{L}_1$.
 We then show that there exists $x_i$, a solution to \eqref{eq:BVPinitial}
 for each energy level. The connecting orbit segment uses 
 $3$ Chebyshev arcs each approximated using $75$ modes, and satisfying
\begin{align*}
\left\| x - \bar x \right\|_{\infty} < 5.0909 \cdot 10^{-7}. \\
\end{align*}
The results are illustrated in Figure \ref{figure:HillsMu0}.

A second set of results, proved similarly, are illustrated in Figure 
\ref{figure:HillsMuJove}, for the Hill Restricted Four Body Problem 
with $\mu = 0.00095$ (the Sun-Jupiter value). Here, we prove the 
existence of transverse heteroclinic connections running both
directions between the $\mathcal{L}_1$ and $\mathcal{L}_2$
planar Lyapunov orbits in the $C = 4$ energy level.  The 
non-zero value of $\mu$ breaks the symmetry of the problem, and 
it is necessary for us to prove both connecting orbits (we cannot
infer one from the other).  This done, we conclude that there is 
an asymmetric heteroclinic cycle, and hence chaotic dynamics 
between the two cycles.  

\appendix
{\tiny

\section{A-posteriori analysis: radii polynomials for nonlinear operators between Banach spaces} \label{Radii}

The goal of the Radii polynomials is to prove that a given operator $T:E \to E$ is a uniform contraction over a subset of $E=E^{(1)}\times \hdots \times E^{(n)}$, a Banach space. The subset provided by this method is a small ball around the numerical approximation of the solution to the problem $T(x)=x$. To do so, we have to define $Y$ and $Z$ such as
\begin{align}\label{eq:Bounds}	 
	\left\| (T(\overa)-\overa)^{(i)} \right\| \leq Y^{(i)} \hspace{0.5cm} &\mbox{and} \hspace{0.5cm} \displaystyle \sup_{b\in B(r)} \left\|\left( DT(\overa +b) \right)^{(i)}\right\|_{\mathcal{B}(E)}  \leq Z^{(i)}(r),
\end{align}
for all $i=1,\hdots,n$. The norm used in theses bounds will depend on the space $E^{(i)}$ and might not be the same for all $i=1,\hdots,n$. Appropriates norms will be explicitly shown for every step of the proof in the corresponding sections. The following theorem is using theses bounds to provide the needed result about $T$.

\begin{theorem}\label{thm:contr}
Let $T$ an operator satisfying the bounds defined in \eqref{eq:Bounds} for a given $\overa$. If \\ \noindent $\left\| Y + Z(r) \right\|_\infty<r$, then for $B:= B_{\overa}(r) $, we have that $T:B \to B$ is a contraction.
\end{theorem}
This theorem provides the existence of a fixed point of the operator $T$ considered for the computation of the bounds \eqref{eq:Bounds}. In order to apply this result, every problem must be written as a Newton like operator. The step to do so depending on the object and the system, it will not be introduced in the general case. Though, it will be explicitly done in every section. \\
For a given operator well defined on some Banach space $E$, the bounds \eqref{eq:Bounds} can be computed analytically for an arbitrary $r$. The goal is now to find a way to obtain $r$ such as the theorem is verified. This is why we introduce the radii polynomials. 
\begin{definition}\label{PolRayon}
Let $Y$ and $Z$ bounds on an operator $T$ as given by \eqref{eq:Bounds}, we define the Radii polynomials:
\[
	p^{(i)}(r):= Y^{(i)} + Z^{(i)}(r) -r, \; i=1,\hdots,n
\]
\end{definition}
One can see that the Radii polynomials depend on $Y$ and $Z$, which are not unique. But the smaller these bounds are, the easier it will be to prove that the operator $T$ is a contraction over a ball around the approximation. 
The following result will show how the Radii polynomials gives us the value of $r$ for which we can apply theorem \eqref{thm:contr}.

\begin{proposition}\label{prop:Radii}
Let $T:E\to E$ and $\overa$ an approximation of its fixed point. Consider $p^{(i)}(r)$, the radii polynomials defined by bounds satisfying \eqref{eq:Bounds} and let
\[
	\mathcal{I}= \lbrace r>0: p^{(i)}(r)<0, \hspace{0.3cm} \forall i \rbrace.
\]
If $\mathcal{I} \neq \emptyset$, then $T:B_{\overa}(r) \to B_{\overa}(r)$ is a contraction for every $r \in \mathcal{I}$.
\end{proposition}
Proofs and further explanation about the background, and basics examples can be found in \cite{HLM}. In the the following section, the background will be applied in order to define the radii polynomials granting validation for approximations of the Floquet exponent for a given periodic orbit of the Lorenz system.

%
%
%
%
%

\section{Some Banach spaces of infinite sequences} \label{sec:BanachSpaces}
Many approximates are done with Fourier series, the coefficients of those series rely in a particular space that we now introduce.
\begin{definition}\label{def:EllnuNorm}
We denote by $\ellnu$ the set of sequences $a=\left(a_k\right)_{k \in \mathbb{Z}}$ such as, for some fixed $\nu \geq 1$, the sum
\[
\sum_{k \in \mathbb{Z}} |a_k| \nu^{|k|} =: \left\| a \right\|_{1,\nu}
\]
converges. One can easily see that this sum defines a norm. 
\end{definition}
\begin{remark}\label{remark:subset}
Let $\nu' \geq \nu$, then
\[
\sum_{k \in \mathbb{Z}} |a_k| \nu^{|k|} \leq \sum_{k \in \mathbb{Z}} |a_k| \nu'^{|k|},
\]
thus if $a\in \ell_{\nu'}^1$ we have $a \in \ellnu$. Since it is true for every element of $\ell_{\nu'}^1$, we have $\ell_{\nu'}^1 \subset \ellnu$. Furthermore, for $a \in \ell_{\nu'}^1$, one can easily see from last inequality that we have $\|a\|_{1,\nu} \leq \|a\|_{\nu'}$.
\end{remark}
The choice of this space is justified by the fact that it is a Banach algebra under the convolution product. The proof is done in the following lemma.
\begin{lemma}\label{LemmaNorm}
For $a,b \in \ellnu$, $\left\| a \ast b \right\|_{1,\nu} \leq \left\| a \right\|_{1,\nu} \left\| b \right\|_{1,\nu} $ 
\end{lemma}
This approximate will be really useful for approximation including convolution products. For every objects studied in this paper, convolution products of sequences from different spaces will occur, the following remark will be important to understand that the estimates from Lemma \ref{LemmaNorm} still stand in a slightly modified version.
\begin{remark}\label{remark:boundNuNuprime}
Let $a \in \ellnu$ and $b \in \ell_{\nu'}^1$. Let $\tilde{\nu} \bydef \min\lbrace \nu, \nu' \rbrace$, then
\begin{equation*}
  \left\| a \ast b \right\|_{\tilde{\nu}} \leq \left\| a \right\|_{\tilde{\nu}} \left\| b \right\|_{\tilde{\nu}} \leq \left\| a \right\|_{1,\nu} \left\| b \right\|_{1,\nu'}.
\end{equation*}
Convolution terms involving coefficients from the previous proof will occur in some of the steps described in this paper. However, different values of $\nu$ might be used for every proof. Thus, in order to make sure that every estimates are valid we must use a smaller value of $\nu$ than the one for the previous proof so we can assume that the estimate is valid in the right space.
\end{remark}

Furthermore, the dual space is isometrically isomorphic to a space which is well known and is easy to work with.
\begin{definition}
$\ell_\nu^\infty = \left\lbrace c=\lbrace c_n\rbrace_{n\in \mathbb{Z}} : c_k \in \mathbb{C}, \: \forall k \in \mathbb{Z}, \mbox{ and }\left\| c \right\|_{\infty,\frac{1}{\nu}} < \infty \right\rbrace$, where $\left\| c \right\|_{\infty,\frac{1}{\nu}}= \displaystyle \sup_{n \in \mathbb{Z}} \frac{|c_n|}{\nu^{|n|}}$.
\end{definition}
This space is also useful to bound sums such as in the following lemma.
\begin{lemma}\label{lemmaForConvo}
Suppose that $a \in \ellnu$ and $c \in \ell_\nu^\infty$. then
\[
\left| \sum_{k \in \mathbb{Z}} c_k a_k \right| \leq \sum_{k \in \mathbb{Z}} |c_k| |a_k| \leq \|c \|_{\infty,\frac{1}{\nu}} \|a \|_{1,\nu}
\]
\end{lemma}
The proof is straightforward, thus it is left to the reader to verify that this lemma stands. It is useful to get sharp bounds on convolution products when it involves a finite dimension approximation, which can be seen as an element of $\ell_\nu^\infty$. The following theorem states that we can study the dual of $\ellnu$ spaces through the $\ell_\nu^\infty$ space. Its proof can be found in \cite{HLM}.
\begin{theorem}
For $\nu \geq 1$, the space $\ell_\nu^\infty$ is isometrically isomorphic to $(\ellnu)^*$.
\end{theorem}
This result gives us the following equality.
\begin{equation}\label{egalityNorm}
\sup_{\left\| a \right\|_{1,\nu}=1} \left| \sum_{n\in \mathbb{Z}} c_na_n \right|= \left\| \ell_c \right\|_{(\ellnu)^*}=\left\| c \right\|_{\infty,\frac{1}{\nu}} = \sup_{n\geq 0} \frac{|c_n|}{\nu^n}.
\end{equation}
Using these kind of sequences will also leads us to infinite operator, whose norm will be evaluated using the following corollary.

\begin{corollary}\label{corollary:BoundEventuallyDiagonal}
Let $A^{(m)}$ a finite matrix of size $2m-1 \times 2m-1$ with complex values entries and $\lbrace u_n \rbrace_{|n|\geq m}$ a bi-infinite sequence of complex numbers such as $|u_n| \leq |u_m|$ for every $|n| \geq m$. 
Let $a=(a_k)_{k\in \mathbb{Z}} \in \ellnu$ and $a^{(m)}=(a_{-m+1},\hdots,a_{m-1}) \in \mathbb{C}^{2m-1}$ its finite dimension projection. Define the linear operator $A$ by
\begin{align*}
 (Aa)_k= \begin{cases}
          (A^{(m)}a^{(m)})_k, & \mbox{ if } |k|< m \\
          u_ka_k, & \mbox{ if } |k|\geq m.
         \end{cases}
\end{align*}
Then $A \in B(\ellnu, \ellnu)$ is bounded and
\[
\left\| A \right\|_{B(\ellnu, \ellnu)} \leq \max(K,|u_m|),
\]
where $K \bydef \displaystyle \max_{|n| < m} \frac{1}{\nu^{|n|}} \sum_{|k|<m} |A_{k,n}|\nu^{|k|}$.
\end{corollary}

The nature of some the objects studied in this work leads to the construction of higher dimensional expansions with one non-periodic direction. We assume again that the series represents an analytic function, which lead to the following definition.

\begin{definition}\label{def:FourierTaylorSpace}
We denote by $X_\nu$ the space of all 
\[
x= \left\{ x_\alpha \in \ellnu : \alpha=0,1,2,\hdots \right\}
\] 
such that
\[
\|x\|_{X_\nu} \bydef \sum_{\alpha=0}^\infty \sum_{k \in \mathbb{Z}} \left| a_{\alpha,k} \right| \nu^{|k|} =   \sum_{\alpha=0}^\infty \left\| a_{\alpha} \right\|_{1,\nu} < \infty.
\]
\end{definition}

An element $x \in X$ will be the result of the expansion of a Taylor series with periodic coefficients, that are themselves expended using a Fourier series. This space also becomes a Banach algebra under the product defined below.

\begin{definition}\label{def:CCproduct}
For $a,b \in X$, we define the Cauchy-convolution product $\star: X \times X \to X $ by
\[
(a \star b)_{\alpha,k} = \displaystyle \sum_{\substack{\alpha_1 + \alpha_2 = \alpha\\ \alpha_1,\alpha_2 \geq 0}} \sum_{\substack{k_1+k_2 = k\\ k_1,k_2 \in \mathbb{Z}}} a_{\alpha_1,k_1}b_{\alpha_2,k_2}.
\]
for all $k \in \mathbb{Z}$ and $\alpha = 0,1,2,\hdots$. The special case where the value $\alpha_1=0$ or $\alpha_2=0$ are omitted from the sum arise in some circumstances of interested, so that this scenario will be given its own notation $\hat \star: X \times X \to X$. It follows that for all $k \in \mathbb{Z}$ and $\alpha \in \mathbb{Z}^+$, the product is given by
\[
(a~ \hat \star ~b)_{\alpha,k} = \displaystyle \sum_{\substack{\alpha_1 + \alpha_2 = \alpha\\ \alpha_1,\alpha_2 > 0}} \sum_{\substack{k_1+k_2 = k\\ k_1,k_2 \in \mathbb{Z}}} a_{\alpha_1,k_1}b_{\alpha_2,k_2}.
\]
\end{definition}

%
%
%
%
%
%
%

\section{Cauchy Bounds}\label{sec:CauchyBound}
\begin{lemma}[Bounds for Analytic Functions given by Fourier Series] \label{lem:cauchyBounds_Fourier}
Suppose that $\{a_n \}_{n \in \mathbb{Z}}$ is a two sided sequence of complex numbers,
that $\nu > 1$, that $\omega > 0$, and that  
\[
\| a\| _{1,\nu} = \sum_{n \in \mathbb{Z}} |a_n| \nu^{|n|}  < \infty.
\]
Let 
\[
f(z) := \sum_{n \in \mathbb{Z}} a_n e^{i \omega k z}.
\]
For $r > 0$ let 
\[
A_r := \{z \in \mathbb{C} \, | \, |\mbox{{\em imag}}(z)| < r \},
\]
denote the open complex strip of width $r$.
\begin{itemize}
\item[(1)] \textbf{$\ell_\nu^1$ bounds imply $C^0$ bounds:} 
The function $f$ is analytic on the strip $A_r$ with $r = \ln(\nu)/\omega$,
continuous on the closure of $A_r$, and satisfies
\[
\sup_{z \in A_r} |f(z)| \leq \| a\|_{1,\nu}.
\]
Moreover $f$ is $T$-periodic with $T = 2 \pi/\omega$. 
\item[(II)] \textbf{Cauchy Bounds:} Let $0 < \sigma < r$ and 
\[
\tilde \nu = e^{\omega (r - \sigma)}.
\]
Define the sequence $b = \{b_n \}_{n \in \mathbb{Z}}$ by 
\[
b_n = i \omega n a_n.
\]
Then 
\[
\| b \|_{\tilde \nu}^1 \leq \frac{1}{e \sigma} \|a\|_{1,\nu},
\]
and
\[
\sup_{z \in A_{r - \sigma}} |f'(z)| \leq \frac{1}{e \sigma} \| a\|_{1,\nu}.
\]
\end{itemize}
\end{lemma}

\begin{proof}
For $\nu > 1$, let $r = \ln(\nu)/\omega$ (i.e. $\nu  = e^{\omega r}$) and consider $z \in A_r$.
We have that 
\begin{eqnarray}
|f(z)| &\leq& \sum_{n \in \mathbb{Z}} |a_n| \left| \left(e^{i \omega z}\right)^{n} \right| \\
&\leq& \sum_{n \in \mathbb{Z}} |a_n| \left(e^{\omega \left|\mbox{imag}(z) \right|}\right)^{|n|} \\
&\leq& \sum_{n \in \mathbb{Z}} |a_n| \nu^{|n|} \\
&=& \| a \|_{1,\nu}.
\end{eqnarray}
It follows that $f$ is analytic as the Fourier series converges absolutely and uniformly in $A_r$.
Continuity at the boundary of the strip also follows from the absolute summability of the series.

From the fact that $f$ is analytic in $A_r$ it follows that the derivative $f'(z)$ exists for any $z$ in the interior
of $A_r$ and that 
\[
f'(z) = \sum_{n \in \mathbb{Z}} i \omega n a_n e^{i \omega n z} = \sum_{n \in \mathbb{Z}} b_n e^{i \omega n z}.
\]
However the fact that $f$ is uniformly bounded on $A_r$ does not imply uniform 
bounds on $f'$.  Indeed it may be that $f'$ has singularities at the boundary.  
In order to obtain uniform bounds on the derivative we give up a portion of the width of the domain. More precisely we 
consider the supremum of $f'$ on the strip $A_{r - \sigma}$ with $0 < \sigma < r$.

First recall that $\nu = e^{\omega r}$ so that $\tilde \nu = e^{\omega(r - \sigma)}$.
We now define the function $s \colon \mathbb{R}^+ \to \mathbb{R}$ by 
\[
s(x) := x \alpha^x,
\]
with $\alpha = \tilde{\nu} / \nu = e^{-\omega \sigma} < 1$.
Note that $s(x) \geq 0$ for all $x \geq 0$,  $s(0) = 0$, and that $s(x) \to 0$ as $x \to \infty$.
Moreover $s$ is bounded and attains its maximum at 
\[
\hat x = \frac{1}{\omega \sigma}, 
\]
as can be seen by computing the critical point of $s$.  
Then we have the bound
\[
s(x) \leq s(\hat x) = \frac{1}{e \omega \sigma}.
\]
From this we obtain that
\begin{eqnarray*}
\| b\|_{1,tilde \nu} &=& \sum_{n \in \mathbb{Z}} |b_n| \tilde{\nu}^{|n|} \\
&=& \sum_{n \in \mathbb{Z}} \omega |n| |a_n| \tilde{\nu}^{|n|} \frac{\nu^{|n|}}{\nu^{|n|}} \\
&=& \sum_{n \in \mathbb{Z}} \omega |n| \left(\frac{\tilde \nu}{\nu}  \right)^{|n|} |a_n| \nu^{|n|}  \\
&=& \sum_{n \in \mathbb{Z}}  \omega s(|n|) |a_n| \nu^{|n|} \\
&\leq& \sum_{n \in \mathbb{Z}}  \omega \frac{1}{e \omega \sigma} |a_n| \nu^{|n|} \\ 
&=& \frac{1}{e \sigma} \| a\|_{1,\nu}.
\end{eqnarray*} 
It now follows by $(I)$ that if $z \in A_{r - \sigma}$ then
\[
|f'(z)| \leq  \|b\|_{1,\tilde \nu} \leq \frac{1}{e \sigma} \| a\|_{1,\nu},
\]
as desired.
\end{proof}

\begin{lemma}[Bounds for Analytic Functions given by Taylor Series] \label{lem:cauchyBounds_Taylor}
Let $\nu > 0$ and suppose that $\{a_n \}_{n \in \mathbb{N}}$ is a one sided sequence of complex numbers 
with  
\[
\| a\| _{1,\nu} = \sum_{n =0}^\infty |a_n| \nu^n  < \infty.
\]
Define 
\[
f(z) := \sum_{n=0}^\infty a_n z^n,
\]
and let
\[
D_\nu := \{z \in \mathbb{C} \, | \, |z| < \nu \},
\]
denote the complex disk of radius $\nu > 0$ centered at the origin.
\begin{itemize}
\item[(1)] \textbf{$\ell_\nu^1$ bounds imply $C^0$ bounds:} 
The function $f$ is analytic on the disk $D_\nu$,
continuous on the closure of $D_\nu$, and satisfies
\[
\sup_{z \in D_\nu} |f(z)| \leq \| a\|_{1,\nu}.
\] 
\item[(II)] \textbf{Cauchy Bounds:} Let $0 < \sigma \leq 1$ and 
\[
\tilde \nu = \nu e^{- \sigma}.
\]
Then 
\[
\sup_{z \in D_{\tilde \nu}} |f'(z)| \leq \frac{1}{\nu \sigma} \|a\|_{1,\nu}.
\]
\end{itemize}
\end{lemma}

\begin{proof}
The proof is similar to the proof of Lemma \ref{lem:cauchyBounds_Fourier}, the difference being the 
estimate of the derivative.  So, since $f$ is analytic in $D_\nu$ we know that 
\[
f'(z) = \sum_{n=1}^\infty n a_n z^{n-1},
\]
for all $z \in B_\nu$, and again we will trade in some of the domain in order to obtain 
uniform bounds.  To this end 
choose $0 < \sigma \leq 1$ and define $\tilde \nu = \nu e^{-\sigma}$.  As in the Fourier case we 
consider a function $s \colon [0, \infty) \to \mathbb{R}$ defined by 
\[
s(x) := x e^{- \sigma x},
\]
and have that $s$ is a positive function with
\[
s(x) \leq \frac{1}{e \sigma}.
\]
Then for any $z \in D_{\tilde \nu}$ we have that 
\begin{eqnarray*}
|f'(z)| &=& \sum_{n=1}^\infty n |a_n| |z|^{n-1} \\
&\leq& \sum_{n=1}^\infty n |a_n|  \frac{\nu}{\nu}  | \nu e^{-\sigma} |^{n-1}  \\
&\leq& \sum_{n=1}^\infty \left( \frac{e^{\sigma}}{\nu} n e^{-\sigma n}  \right) |a_n| \nu^n \\
&\leq& \sum_{n=1}^\infty \left( \frac{e}{\nu} n e^{-\sigma n}  \right) |a_n| \nu^n \\
&\leq& \sum_{n=0}^{\infty} \frac{e}{\nu} s(n) |a_n|\nu^n \\
&\leq& \sum_{n=0}^\infty \frac{e}{\nu} \frac{1}{e \sigma} |a_n| \nu^n \\
&=& \frac{1}{\nu \sigma} \|a\|_{1,\nu}^1.
\end{eqnarray*}
\end{proof}

}

\bibliographystyle{plain}
\bibliography{references} 

\end{document}